\documentclass[]{siamonline190516}

\usepackage{lipsum}
\usepackage{amsfonts}
\usepackage{graphicx}
\graphicspath{{./Figures/}}	
\usepackage{epstopdf}
\usepackage{algorithmic}
\ifpdf
  \DeclareGraphicsExtensions{.eps,.pdf,.png,.jpg}
\else
  \DeclareGraphicsExtensions{.eps}
\fi

\usepackage{amsopn}

\usepackage{dsfont}  	
\usepackage{newtxtext,newtxmath}
\usepackage{microtype}          

\usepackage[textsize = tiny]{todonotes}
\usepackage[normalem]{ulem}              
\usepackage{mathrsfs}
\usepackage{units,dsfont}
\usepackage{mathtools}
\usepackage[square,numbers]{natbib} %

\usepackage[shortlabels, inline]{enumitem}
\setlist[enumerate, 1]{label=(\roman*), leftmargin=.5in}
\setlist[itemize]{leftmargin=.5in}

\newsiamremark{remark}{Remark}
\newsiamremark{hypothesis}{Hypothesis}
\crefname{hypothesis}{Hypothesis}{Hypotheses}
\newsiamthm{claim}{Claim}
\newsiamthm{example}{Example}
\newsiamremark{assumption}{Assumption}

\allowdisplaybreaks

\headers{
Wasserstein Sensitivity of Risk and Uncertainty Propagation}
{O. G. Ernst, A. Pichler and B. Sprungk}

\title{Wasserstein Sensitivity of Risk and Uncertainty Propagation}

\author{Oliver G. Ernst\thanks{Department of Mathematics, TU Chemnitz, Germany 
  (\email{oernst@math.tu-chemnitz.de}, 
   \email{alois.pichler@math.tu-chemnitz.de}).
}
\and Alois Pichler\footnotemark[1]\textsuperscript{\;\,}\thanks{DFG, German Research Foundation~-- Project-ID 416228727 -- SFB~1410}
\and Björn Sprungk\thanks{TU Bergakademie Freiberg, Faculty of Mathematics and Computer Science, Freiberg, Germany
  (\email{bjoern.sprungk@math.tu-freiberg.de}).
}}

\DeclareMathAlphabet{\pazocal}{OMS}{zplm}{m}{n}

\DeclareMathOperator{\one}{{\mathds 1}} 	
\DeclareMathOperator{\var}{{\mathsf{Var}}}				
\DeclareMathOperator{\cov}{cov}				
\DeclareMathOperator{\Risk}{ {\rho} }	
\DeclareMathOperator{\VaR}{{V@R}}			
\DeclareMathOperator{\AVaR}{{AV@R}}			
\DeclareMathOperator{\EVaR}{{EV@R}}			

\DeclareMathOperator*{\supp}{supp}

\DeclareMathOperator*{\essinf}{ess\,inf}
\DeclareMathOperator*{\esssup}{ess\,sup}

\DeclareMathOperator{\diam}{diam}
\DeclareMathOperator{\tr}{tr}

\newcommand{\mc}{\mathcal} 
\DeclareMathOperator{\mcX}{\mathcal X}

\newcommand{\map}{\mathcal S}

\newcommand{\e}{\mathrm e}					
\newcommand{\dd}{\mathrm d}					
\newcommand{\ev}[1]{ {\operatorname{\mathsf E}} 	  \left[#1 \right]}
\newcommand{\evm}[2]{ {\operatorname{\mathsf E}}_{#1} \left[#2 \right]}

\newcommand{\prob}{{{\mathsf P}}}
\newcommand{\probalt}{{{\mathsf Q}}}

\DeclareMathOperator{\Lip}{Lip}
					
\newcommand{\vZ}{\boldsymbol Z}					

\newcommand{\bbR}{\mathbb R}

\ifpdf
\hypersetup{
	pdftitle={
	Wasserstein sensitivity of Risk and Uncertainty Propagation},
	pdfauthor={Oliver G. Ernst, Alois Pichler, Björn Sprungk}
}
\fi

\newcommand{\rev}[1]{#1} 

\begin{document}
\maketitle

\begin{abstract}
	When propagating uncertainty in the data of differential equations, the probability laws describing the uncertainty are typically themselves subject to uncertainty. 
	We present a sensitivity analysis of uncertainty propagation for differential equations with random inputs to perturbations of the input measures.
	We focus on the elliptic diffusion equation with random coefficient and source term, for which the probability measure of the solution random field is shown to be Lipschitz-continuous in both total variation and Wasserstein distance. 
	The result generalizes to the solution map of any differential equation with locally Hölder dependence on input parameters.
	In addition, these results extend to Lipschitz continuous quantities of interest of the solution as well as to coherent risk functionals of these applied to evaluate the impact of their uncertainty. 
	Our analysis is based on the sensitivity of risk functionals and pushforward measures for locally Hölder mappings with respect to the Wasserstein distance of perturbed input distributions.
	The established results are applied, in particular, to the case of lognormal diffusion and the truncation of series representations of input random fields.
\end{abstract}



\section{Introduction}

A fundamental task in uncertainty quantification (UQ) for models in the physical sciences is the solution of differential equations with random inputs.
These account in a probabilistic fashion for uncertainty in the data of a differential equation potentially arising in coefficient functions, source terms, initial and boundary data as well as the domain on which the problem is posed.
Within the broadening discipline of UQ this task is referred to as \emph{uncertainty propagation}, or simply \emph{forward UQ}. 
Once probability laws for the uncertain data have been identified, the solution is determined as a stochastic process (random function, random field), and often functionals of the solution known as \emph{quantities of interest (QoI)} and their statistics are of primary interest in the analysis. 
The statistical post-processing of (random) quantities of interest aims to extract useful information from the results of the computation for objectives such as optimization or decision support.
Besides statistical moments or the probability of specific events, statistics known as \emph{risk measures} or \emph{risk functionals} are of particular interest in evaluating the impact of uncertainty in the output quantities on the application under study.

This work investigates sensitivity of the results of an uncertainty propagation analysis with respect to perturbation of input probability measures.
\rev{The setting is rather general and requires only a locally H\"older continuous forward mapping.
However, we follow established practice and employ as a running example the well-studied model problem of an elliptic diffusion equation on a (fixed) bounded domain with uncertain coefficient and source term, both modeled as random fields.}
The dependence of the solution of a random diffusion equation on such data has been considered previously, e.g., by 
Babuška et al.\ 
\cite[Section~2.3]{BabuskaEtAl2004}, 
\cite[Section~2.3]{BabuskaEtAl2005} and Charrier 
\cite[Section~4]{Charrier2012}.
In these analyses of sensitivity, the random inputs and outputs are treated as (function-valued) random variables and their perturbations measured by Bochner space norms.
In this setting, the random PDE is equivalent to a PDE with a (high- or countably infinite-dimensional) parameter, which is represented by a vector or sequence of independent basic random variables serving as a coordinate system in the random dimensions.
Their variation is naturally measured by $L^p$ norms weighted by the densities of the basic random variables.
In this way, the uncertainty propagation problem reduces to the numerical solution of parametric PDEs, resulting in an entirely deterministic problem formulation.

In UQ analysis it is typically the probability distribution of the random outcomes that is of primary interest.
At the same time, the precise probability distribution of the random inputs is generally unavailable and possibly guessed, elicited from expert opinion or the result of conditioning a prior probability distribution on observations as in the Bayesian formulation of inverse problems.
There is, therefore, considerable uncertainty associated with the very probability measures used to model uncertainty in the inputs.
\rev{Regarding this problem the sensitivity of Bayesian inference with respect to the choice of the prior measure has been recently examined in \cite{Sprungk2020}.
In this work we provide an analogous study focusing on the propagation of uncertainty and risk.}

In what follows we derive basic sensitivity results of Hölder type for
\begin{enumerate*}
\item the distribution of the random solution of \rev{a differential equation with random coefficients}, 
\item the distribution of Lipschitz continuous quantities of interest of this solution, and 
\item general risk functionals of quantities of interest---in each case with respect to perturbations of
\end{enumerate*} 
the underlying probability measure of the random inputs.
he results further extend to quantities of interest which are only \emph{locally} Lipschitz.

Among the vast variety of distances and divergences for probability measures (see, e.g., Gibbs and Su \cite{GibbsSu2002} or  Rachev \cite{Rachev}), we focus on the Wasserstein metric for measuring perturbations of distributions for the following reasons:
\begin{enumerate}
	\item The Wasserstein distance metrizes weak convergence, i.e., convergence of measures is ensured by testing convergence for appropriate test functions. 
	\item The Wasserstein metric provides a well-defined distance which extends to mutually singular probability measures (as opposed to, e.g., the total variation or Hellinger metrics). 
	This is important for the  UQ setting considered here since probability measures on infinite-dimensional function spaces tend to be mutually singular.
	\item By convex duality, the Wasserstein distance allows sharp lower and upper bounds for the problem in mind. Here we develop upper bounds in detail.
	\item In particular, the $p$-Wasserstein distance of two probability measures (on a normed space) is given by the infimum of the $L^p$-distance between any two random variables following these distributions. 
	Thus, to bound the Wasserstein distance for perturbed random field models, existing results on the Lebesgue norm distance of such random variables (as in 
	Babuška et al.\ \cite{BabuskaEtAl2004,BabuskaEtAl2005}, 
	Charrier \cite{Charrier2012}) can be used.
	\item Finite convex combinations of Dirac measures form a dense subset of all probability measures with respect to Wasserstein distance (cf.\ Bolley \cite{Bolley2008}).
	This yields, in particular, convergence of empirical approximations to the true distribution in the large-sample limit which is important for applications in, for instance, finance and insurance.
\end{enumerate}

Our main result is a general statement of Hölder continuity of output probability measures in terms of the input probability measures in an uncertainty propagation problem. 
This holds in total variation distance assuming only measurability of the propagation map \rev{(Proposition~\ref{propo:robust_TV})} and in Wasserstein distance under the additional assumption of global Hölder continuity \rev{(Theorem~\ref{theo:Lip})}.
When the propagation map includes a PDE solution operator, typically only local Hölder continuity holds.
In this case we obtain a local Hölder continuity result under an additional integrability condition \rev{(Theorem~\ref{theo:LocLip})}.
We apply these results to carry out a detailed analysis for the stationary diffusion equation with random coefficient field, covering in particular stability with respect to truncation of the input field \rev{(Section~\ref{sec:StrawberryFields})} as well as the challenging case of lognormal conductivity \rev{(Corollary~\ref{coro:lognormal})}.
Because of their increasing relevance to UQ analyses and because the sensitivity results presented in Section~\ref{sec:UQ_sens} extend naturally to their analysis, we also devote a section to the sensitivity of risk functionals.
\rev{
	We establish local H\"older continuity of coherent risk functionals for locally H\"older quantities of interest with respect to the Wasserstein distance of the input probability measures (Theorem~\ref{thm:12}).
	Combined with the results in Section~\ref{sec:UQ_sens}, this implies local H\"older sensitivity of risk functionals associated with suitable quantities of interest of solutions to random differential equations with respect to the underlying input probability distributions (Corollary~\ref{thm:UQ}).


We provide a brief (incomplete) account of previous work on risk functionals, as these may not be so well known among the UQ community.} 
In use for decades in mathematical finance (see Artzner et al.\ \cite{Artzner1997, Artzner1999}) 
and even earlier in insurance (cf.\ Denneberg \cite{Denneberg1989} and Deprez and Gerber \cite{GerberDeprez}) to quantify the---typically negative---impact of random outcomes, risk functionals have recently drawn increased attention also in uncertainty analysis for engineering applications (see the survey article of Rockafellar and Royset \cite{RockafellarRoyset2015} for a number of highly relevant use cases).
Specifically, recent work employing risk measures in a UQ context covers
risk-averse optimization under uncertainty 
(Conti et al.\ \cite{Schultz2011}, Harbrecht \cite{Harbrecht2008}),
integrating aleatoric and epistemic uncertainty (Chowdhary and Dupuis \cite{Dupuis2013}),
energy economics (Halkos and Tsirivis \cite{Halkos_2019}, Pflug and Gross \cite{Pflug2016})
reduced-order modelling (Heinkenschloss et al.\ \cite{HeinkenschlossEtAl2018,HeinkenschlossEtAl2020}),
risk-averse PDE constrained optimization under uncertainty
(Kouri and Surowiec \cite{Kouri2016, Kouri2018},
Geiersbach and Wollner \cite{Geiersbach2019}),
stochastic dynamics (Dupuis et al.\ \cite{Dupuis2016})
and production planning (Göttlich and Knapp \cite{GoettlichKnapp2020}).

\paragraph{Outline of the paper}
The following section provides the mathematical setting of random PDEs, quantities of interest and dependence on parameters.
Section~\ref{sec:UQ_sens} presents a number of general sensitivity results and applies these to the elliptic diffusion problem in detail.
Section~\ref{sec:Risk} introduces risk functionals.
To analyze risk functionals under changing probability measures we investigate these on product spaces and establish new continuity results with respect to the Wasserstein distance in Section~\ref{sec:Bivariate}.
Section~\ref{sec:Summary} concludes and some technical material is included in the appendix.

\section{Partial Differential Equations With Random Data} \label{sec:Problem}

%
%
%
%
%

Consider a general, possibly nonlinear partial differential equation (PDE) given in the form
\begin{equation}\label{eq:pde}
	\mc F\big(u; (a,f)\big) = 0,
\end{equation} 
where $u$ denotes the solution belonging to a suitable function space $\mc U$, $a$ represents one or more coefficients of the differential operator and $f$ a forcing term. 
Appropriate boundary or initial conditions are assumed to be implied by the domain of $\mc F$.
We refer to $a$ and $f$ as the \emph{data} of the PDE and remark that this could be extended to initial or boundary data without changing the basic approach.
We assume that~$a$ and~$f$ belong to metric spaces (constraints such as nonnegativity may require generalizing beyond linear function spaces) which we denote by $\mc X_1$ and $\mc X_2$, respectively.
Moreover, we consider $\mc X = \mc X_1 \times \mc X_2$ equipped with a product metric as the \emph{data space}.
We assume that for any data $(a,f) \in \mc X$ the PDE~\eqref{eq:pde} possesses a unique solution $u \in \mc U$
and denote the parameter-to-solution map by $\mc S\colon \mc X \to \mc U$, $(a,f)\mapsto \mathcal S(a,f) := u$.
We illustrate this general setting for a common model problem in UQ which will serve as running example throughout the paper.

\paragraph{Elliptic diffusion problem}
Given a bounded Lipschitz domain $D \subseteq \mathbb R^d$, $d=2$, $3$, we consider the boundary value problem
\begin{alignat*}{2}
	\mc F(u;a) 
	=
	\nabla\cdot\left(a\,\nabla u\right) - f
	& = 0 & \quad & \text{ on }D \subset\mathbb R^d,\\
	u & =0 &  & \text{ along }\partial D,
\end{alignat*}
with a scalar diffusion coefficient $a\colon D \to (0,\infty)$ and source term $f\colon D \to \bbR$.
In the standard weak formulation of the problem, we seek the unique solution $u \in H_0^1(D) =: \mc U$ of the variational equation
\begin{align} \label{eq:BVP}
\tag{BVP}
	( a\,\nabla u,\nabla v)_{L^2(D)}
	=
	( f,v)_{L^2(D)} &
	\quad\text{ for all } v \in H_0^1(D);
\end{align}
where $(\cdot,\cdot)_{L^2(D)}$ denotes the inner product in $L^2(D)$. 
Assuming $a\in L^\infty(D)$ to be uniformly positive
\[
	a(x) \ge a_{\min} > 0 \qquad \text{for } x\in D \text{ a.e.}
\]
and $f\in L^2(D)$, there exists a unique solution to~\eqref{eq:BVP} such that 
(cf.\ \cite[Theorem~2.7.7]{BrennerScott2008}, \cite[Theorem~6.6.2]{OdenDemkowicz2018})
\begin{equation}\label{eq:8}
	\|u\|_{H_0^1(D)} \leq \frac{c_P}{a_{\min}} \|f\|_{L^2(D)},
\end{equation}
with $c_P$ the Poincaré constant of the domain $D$.
Thus, denoting the set of admissible diffusion coefficients and source terms by
\[
	\mc X_1 := L^\infty_+(D) = \{ a \in L^\infty(D) \colon \essinf a > 0 \}, \qquad
	\mc X_2 := L^2(D),
\]
respectively, we obtain the (nonlinear) solution operator $\mc S\colon \mc X \to \mc U$ mapping the pair $(a,f) \in L^\infty_+(D) \times L^2(D)$ to the corresponding solution $u=S(a,f)$ of~\eqref{eq:BVP}.


\subsection{Dependence on Data}
In many applied settings model parameters such as $a$ and $f$ in~\eqref{eq:pde} are  not known precisely and can often only be estimated up to some remaining uncertainty.
However, for many PDE models the solution depends continuously on the problem data, i.e., the solution operator~$\mc S$ mapping the data $(a,f)$ to the solution $u$ of~\eqref{eq:pde} is continuous w.r.t.\@ suitable metrics or norms for $a$, $f$ and $u$.
In the remainder, we distinguish two types of continuity for $\mc S$ formalized by the following two assumptions, where $\mc X$ denotes the metric space of all admissible pairs $x = (a,f)$ and $\mc U$ a metric space containing the solution $u$.
\begin{assumption} \label{assum:Lip}
The mapping $\mc S\colon \mc X \to \mc U$ is H\"older continuous with exponent $\beta \in (0,1]$ and constant $C_{\mc S}<\infty$ w.r.t.\@ a metric $d_{\mc X}$ on $\mc X$ and $d_{\mc U}$ on $\mc U$, respectively, i.e.,
\[
	d_{\mc U}\left( \mc S(x_1), \mc S(x_2) \right)
	\leq
	C_{\mc S} \ d_{\mc X}\left(x_1, x_2 \right)^\beta
	\qquad
	\text{for all } x_1,x_2 \in \mc X.	
\]
\end{assumption}
\begin{assumption} \label{assum:LocLip}
The mapping $\mc S\colon \mc X \to \mc U$ is locally H\"older continuous with exponent $\beta \in (0,1]$ w.r.t.\@ a metric $d_{\mc X}$ on $\mc X$ and $d_{\mc U}$ on $\mc U$, respectively, i.e., there exists a \rev{reference}~$x_0 \in \mc X$ and a nondecreasing function $C_{\mc S} \colon [0,\infty) \to [0,\infty)$ such that for any radius~$r>0$
\[
	d_{\mc U}\left( \mc S(x_1), \mc S(x_2) \right)
	\leq
	C_{\mc S}(r)\ d_{\mc X}\left(x_1, x_2 \right)^\rev{\beta}
	\qquad
	\forall x_1,x_2 \in B_r(x_0),
\]
where $B_r(x_0) := \{x \in \mc X\colon d_{\mc X}(x,x_0) \leq r\}$ denotes a (closed) ball of radius $r$ around $x_0$ in $\mc X$.
\end{assumption}
The first assumption is, in general, the stronger one.
If, however, the metric space $\mc X$ is bounded w.r.t.~$d_{\mc X}$, then both assumptions are equivalent.
We next verify the second assumption for our running example.

\paragraph{The elliptic problem~\eqref{eq:BVP}}
Merely local Lipschitz continuity is typical, even for the simple linear diffusion equation due to the nonlinear dependence of the solution $u$ on the coefficient $a$. 
For fixed $a$, however, the solution depends linearly on $f$ and we have the stronger global Lipschitz property.
It is well-known (cf.\ \cite[Lemma~2.1]{BonitoEtAl2017}) that the solution operator~$\mc S$ of~\eqref{eq:BVP} satisfies
\begin{equation}\label{eq:Problem}
	\|\mathcal{S}(a_2,f_2)-\mathcal{S}(a_1,f_1)\|_{H_0^1(D)} 
	\leq
	\frac{c}{a_{2,\text{min}}}\|f_2-f_1\|_{L^2(D)}
	+
	\frac{\|\mathcal S (a_1,f_1)\|_{H_0^1(D)}}{a_{2,\text{min}}}\|a_2-a_1\|_{L^{\infty}(D)},
\end{equation}
where $a_{i,\min}$ denotes the essential infimum of the diffusion coefficient, $a_i(x)\ge a_{i,\text{min}}$ ($i=1,\,2$).
Employing the bound~\eqref{eq:8} for $\|\mathcal S (a_1,f_1)\|_{H_0^1(D)}$ we obtain
\begin{equation}\label{eq:Problem_2}
	\|\mathcal{S}(a_2,f_2)-\mathcal{S}(a_1,f_1)\|_{H_0^1(D)} 
	\leq
	\frac{c}{a_{2,\min}}\|f_2-f_1\|_{L^2(D)}
	+
	\frac{c\,\|f_1\|_{L^2(D)}}{a_{1,\min}\ a_{2,\text{min}}}\|a_2-a_1\|_{L^{\infty}(D)},
\end{equation}
which is the basis for the following continuity result  (cf.\ \cite[Theorem~2.46]{LordEtAl2014}).

\begin{proposition}[Local Lipschitz continuity for~\eqref{eq:BVP}] \label{prop:Lipschitz}
Let $ r_a$, $r_f \in(0,\infty)$ be given radii.
Then there holds for any $a_1, a_2 \in L^\infty_+(D)$ with $\|\log a_i \|_{L^\infty} \leq r_a$, $i=1,2$, and any $f_1$, $f_2\in L^2(D)$ with $\|f_i\|_{L^2} \leq r_f$, $i=1,2$, that
\begin{equation}\label{eq:3}
	\left\Vert \mathcal{S}(a_{2,}f_2)-\mathcal{S}(a_1,f_1)\right\Vert_{H_0^1(D)}  
	\le 
	c_a\,\|a_2-a_1\|_{L^{\infty}(D)} + c_f\,\|f_2-f_1\|_{L^2(D)},
\end{equation}
where $c_{a}:=c\,r_f \e^{2r_a}$ and $c_{f}:=c\e^{r_a}$, as well as
\begin{equation} \label{eq:4}
	\left\| \mathcal{S}(a_2,f_2)- \mathcal{S}(a_1,f_1)\right\|_{H_0^1(D)}  
	\le 
	\tilde c_a\,\|\log a_2 - \log a_1\|_{L^{\infty}(D)} + c_f\,\|f_2-f_1\|_{L^2(D)},
\end{equation}
where $\tilde c_a:=c\,r_f\, \e^{3r_a}$ and $c_f$ as above.
\end{proposition}
\begin{proof}
Estimate~\eqref{eq:3} follows from~\eqref{eq:Problem_2} using 
$a_\text{min} := \exp( -r_a) \leq \exp(- \|\log a_i \|_{L^\infty}) \leq \essinf a_i$.
Estimate~\eqref{eq:4} follows from~\eqref{eq:3} combined with
\begin{align*}
	\|a_1 - a_2\|_{L^{\infty}(D)}
	& =
	\|\exp(\log a_2) - \exp(\log a_1)\|_{L^{\infty}(D)}\\
	& \leq
	\exp\big( \max\{\|\log a_1 \|_{L^{\infty}(D)},\, \|\log a_2 \|_{L^{\infty}(D)} \} \big)\
	\|\log a_2 - \log a_1 \|_{L^{\infty}(D)}\\
	& \leq
	\exp\big( r_a\big)\
	\|\log a_2 - \log a_1 \|_{L^{\infty}(D)},
\end{align*}
where we have used local Lipschitz continuity of the exponential function.
\end{proof}
Thus, the solution operator~$\mc S$ for the weak formulation~\eqref{eq:BVP} mapping from $\mc X := L^\infty_+\times L^2(D)$ to $\mc U := H_0^1(D)$ satisfies Assumption~\ref{assum:LocLip} for the metric $d_{\mc U}$ induced by $\|\cdot\|_{H_0^1(D)}$ and 
\begin{equation}\label{equ:BVP_d}
	d_{\mc X}\big( (a_1, f_1), (a_2, f_2) \big)
	:=
	\|\log a_1 - \log a_2\|_{L^\infty(D)} + \|f_1-f_2\|_{L^2(D)}
\end{equation}
with $\beta = 1$ and
\begin{equation}\label{equ:BVP_C}
	C_{\mc S}(r)
	:=
	c (1+r) \e^{3r}, 
	\qquad r \in [0,\infty).
\end{equation}

\subsection{Random Data and Propagation of Uncertainty} \label{sec:Uncertainty}
 
In practice, the data of PDE models~\eqref{eq:pde} such as, e.g., the diffusion coefficient $a \in L_+^{\infty}(D)$ and source term $f\in L^2(D)$ in~\eqref{eq:BVP}, are frequently not known precisely.
This limited knowledge of $a$ and $f$ is often modeled probabilistically, i.e., as random variables with realizations in corresponding metric spaces $\mc X_1$ and $\mc X_2$, respectively: $a(\omega)\in\mc X_1$ and $f(\omega)\in\mc X_2$, $\omega \in \Omega$ with respect to some reference probability space $(\Omega, \mathcal{F}, \prob)$.
Each realization of the random data then gives rise to a random solution $u=u(\omega) \in \mc U$ of~\eqref{eq:pde} where now we require
\begin{equation}\label{eq:randomPDE}
	\mc F\big(u(\omega); (a(\omega), f(\omega))\big) = 0 \qquad \prob\text{-almost surely.}
\end{equation}
This random solution is then given by $u(\omega) = \mc S\big(a(\omega), f(\omega)\big)$ and the continuity of $\mc S$ ensures the measurability of $u\colon \Omega \to \mc U$.
In order to model the random data, we usually assume a distribution for $a$ and $f$, respectively, i.e., a probability measure $\prob_a$ on $\mc X_1$ for the random $a\sim \prob_a$ as well as a probability measure $\prob_f$ on $\mc X_2$ for the random $f\sim \prob_f$.
Here, the metric spaces $\mc X_1$, $\mc X_2$ are equipped with the associated Borel $\sigma$\nobreakdash-algebra.
Often it is reasonable to postulate the independence of $a$ and $f$ which then results in the product measure $\prob = \prob_a \otimes \prob_f$  for the joint distribution on the data space~$\mc X$ representing our uncertainty about the PDE data.
By means of the solution operator this propagates to the probability $\mc S_\ast\prob$ of the random solution $u\sim \mc S_\ast\prob$ where $\mc S_\ast\prob := \prob \circ \mc S^{-1}$ denotes the pushforward measure of $\prob$ under $\mc S\colon \mc X \to \mc U$.
In the following section we study the sensitivity of $\mc S_\ast\prob$ w.r.t.\@ changing or perturbing the input measure $\prob$.
The focus in practice is typically less on the complete random solution $u$ than on specific aspects of it, collectively termed \emph{quantities of interest (QoI)}.
We model these as functionals $\phi\colon \mc U \to\mathbb{R}$.
The composition
\begin{equation}\label{eq:Phi}
	\omega\mapsto\phi\big(u(\omega)\big)
	= 
	(\phi \circ \mathcal{S})\big( a(\omega),\,f(\omega)\big)
\end{equation}	
is then a real-valued random variable, for which we are interested in statistics such as, for example, 
$\ev{\phi\big(u(\omega)\big)}$ or $\var{\phi\big(u(\omega)\big)}$.
Again, the resulting distribution of the random variable $\phi(u)$ is the pushforward measure $(\phi\circ S)_\ast\prob = \phi_*(\mc S_*\prob)$ given that the data follows the law $(a,f) \sim \prob$.

\paragraph{The elliptic problem~\eqref{eq:BVP}}
For the elliptic problem~\eqref{eq:BVP} the random data is given by
\[
	a(\cdot,\omega)\in L_+^{\infty}(D)
	\quad \text{ and } \quad 
	f(\cdot,\omega)\in L^2(D),\quad\omega\in\Omega.
\]
Here, as the Banach space $L^{\infty}$ is not separable, we may occasionally prefer $a(\cdot,\omega)\in C_+(D)$ (strictly positive continuous functions) as the space of outcomes for the log diffusion coefficient instead.
A common stochastic model for the random coefficient~\(a\) is the \emph{lognormal distribution}, for which $\log a$ is a Gaussian random field on $D$ with continuous paths.
Also, for the random~$f$ we may assume a Gaussian measure $\prob_f$.
The random $\log a(\cdot,\omega)$ and $f(\cdot,\omega)$ then result in a random solution $u(\cdot,\omega)$ of~\eqref{eq:BVP}, i.e.,
\begin{alignat}{2}
	-\nabla\cdot\big(a(\cdot,\omega)\,\nabla u(\cdot,\omega)\big) & =f(\cdot,\omega) & \qquad & \text{ on }D\subset\mathbb{R}^d\text{ and}\label{eq:5}\\
	u(\cdot,\omega) & =0 &  & \text{ on }\partial D,\nonumber
\end{alignat}
with $u(\cdot,\omega)\in H_0^1(D)$ with probability~$1$, i.e., almost surely. 
In this setting, the solution
\[
	\omega\mapsto u(\cdot,\omega) = \mathcal{S}\big( a(\cdot,\omega),\,f(\cdot,\omega)\big)
\]
of~\eqref{eq:5} is a random variable taking values in $H_0^1(D)$.
Simple but important examples of corresponding quantities of interest $\phi\colon H_0^1(D)\to\bbR$ include point evaluations $\phi(u):=u(x_0)$ (if defined) or $\phi(u):=\int_{D^{\prime}}u(x)\,\mathrm dx$ for some fixed domain $D^{\prime}\subseteq D$. 
Moreover, suitable nonlinear quantities of interest are
\begin{align*} 
	\phi(u) &:= \int_{D^{\prime}} \bigl|u_0-u(x)\bigr|^2\,\dd x,
	&& u_0 \in L^2(D),\\
\intertext{and}
	\phi(u) &:= \int_{D^{\prime}} \bigl|\nabla u_0 -\nabla u(x)\bigr|^2\,\dd x
	\quad \text { or } \;
	\phi(u) := \|u_0 - u\|_{H^1(D)}
	&& u_0 \in H^1(D),
\end{align*}
which all are Lipschitz in $u$ in any ball of fixed radius in $H_0^1(D)$. 

\section{Sensitivity of Uncertainty Propagation}  \label{sec:UQ_sens}

We now discuss the sensitivity of push-forward measures to perturbations of the measure being propagated, as pertinent for the distribution of random solutions $u \colon \Omega \to \mc U$ of PDEs with random data and related quantities of interest $\phi \colon \mc U\to\mathbb R$ resulting from different random models for the uncertain coefficients $a$ and $f$.
To this end, let $\mc P(\mc X)$ denote the set of all probability measures on the data (metric) space $(\mc X,d_{\mc X})$, equipped with its Borel $\sigma$\nobreakdash-algebra and let $\prob$, $\probalt \in \mc P(\mc X)$ denote two different distributions for the random PDE data.
For the sensitivity analysis we focus on the Wasserstein distance of probability measures for the reasons outlined in the Introduction.
However, as an initial consideration and due to its simplicity we first present a sensitivity result in total variation distance.
This metric is a common one in probability theory and uncertainty quantification and relates to the Hellinger distance---they are topologically equivalent---often used in the analysis of Bayesian inverse problems \cite{Stuart2010} as well as to the Wasserstein distance.
In fact, total variation distance can be expressed in terms of the $1$-Wasserstein distance based on the discrete metric $d(x,y) = \mathds{1} _{x \neq y}$ on $\mc X$.

\subsection{Sensitivity in Total Variation Distance} 
The total variation (TV) distance of measures $\prob, \probalt \in \mc P(\mc X)$ is given by
\[
	d_{\mathrm{TV}}(\prob,\probalt) 
	\coloneqq 
	\sup_{A \subseteq \mathcal X} |\prob(A) - \probalt(A)|,
	\qquad \prob, \probalt \in \mathcal P(\mathcal X),
\]
where the supremum taken is over all measurable subsets of a measurable (e.g., Polish) space~$\mathcal X$.
We first obtain the following simple result concerning the sensitivity in TV distance of \emph{general} pushforward measures.

\begin{proposition} \label{propo:robust_TV}
Let $\prob$, $\probalt \in \mathcal P(\mcX )$ and $\map\colon \mcX \to \mc U$ be measurable.
Then
\[
	d_{\mathrm{TV}}\left(\map_{*}\prob, \map_{*}\probalt\right)
	\leq 
	d_{\mathrm{TV}}(\prob,\probalt).
\]
Moreover, if $\mcX = \mcX_1 \times \mcX_2$ and $\prob$ the independent product of measures $\prob:= \prob_1 \otimes \prob_2$ as well as $\probalt = \probalt_1 \otimes \probalt_2$ with $\prob_i, \probalt_i  \in \mc P(\mcX_i)$, then
	\[
		\max_{i=1,2} d_{\mathrm{TV}}(\prob_i,\probalt_i)
		\leq \,
		d_{\mathrm{TV}}(\prob,\probalt) 
		\leq 
		d_{\mathrm{TV}}(\prob_1,\probalt_1) + d_{\mathrm{TV}}(\prob_2,\probalt_2).
	\]
\end{proposition}
\begin{proof}
The first statement is immediate from
\begin{align*}
	d_{\mathrm{TV}}\left( \map_*\prob, \map_*\probalt \right)
	&=	\sup_{B \subseteq \mathcal U} |\prob(\map^{-1}(B)) - \probalt(\map^{-1}(B))|
	=	\sup_{A \in \sigma(\map)} |\prob(A) - \probalt(A)|
	\leq
	d_{\mathrm{TV}}(\prob,\probalt),
\end{align*}
where $\sigma(\map) = \{\map^{-1}(B) \colon B \subseteq \mathcal U \text{ measurable}\}$ denotes the $\sigma$\nobreakdash-algebra induced by $\map$.
The proof of the second statement anticipates the equivalent characterization of TV distance as the $1$-Wasserstein distance based on the discrete metric $d(x,y) = \mathds{1} _{x \neq y}$ on $\mc X$.
In particular, we have (cf.\ \cite[p.\ 103]{Villani2009})
\[	
	d_\mathrm{TV}(\prob,\probalt) 
	= 
	2 \,
	\inf_{\pi} \iint_{\mathcal X\times \mathcal X} \mathds{1} _{x \neq y} \,\pi(\mathrm dx,\mathrm dy),
\]
where the measures~$\pi$ on $\mc X \times \mc X$ in the infimum have marginals~$\prob$ and~$\probalt$, respectively (cf.\ the \emph{couplings} introduced Section~\ref{sec:Wasserstein_Sens}).
In the case of product measures $\prob= \prob_1 \otimes \prob_2$, $\probalt = \probalt_1 \otimes \probalt_2$ on $\mcX = \mcX_1 \times \mcX_2$ we can exploit for $x = (x_1,x_2), y = (y_1,y_2)$ with $x_1,y_1 \in \mc X_1$ and $x_2,y_2\in\mc X_2$ that
\[
	\mathds{1} _{x \neq y} = \max_{i=1,2} \mathds{1} _{x_i \neq y_i} \leq \mathds{1} _{x_1 \neq y_1} + \mathds{1} _{x_2 \neq y_2},
\]
which in turn yields the upper bound by
\begin{align*}
	\inf_{\pi} \iint_{\mathcal X\times \mathcal X} \mathds{1} _{x \neq y} \,\pi(\mathrm dx,\mathrm dy)
	& \leq
	\inf_{\pi_1 \otimes \pi_2}
	\iint_{\mathcal X\times \mathcal X} \left(\mathds{1} _{x_1 \neq y_1} + \mathds{1} _{x_2 \neq y_2}\right)
	\,(\pi_1 \otimes \pi_2) (\mathrm dx, \mathrm dy)\\[.3em]
	& = 
	\inf_{\pi_1}
	\iint_{\mathcal X_1\times \mathcal X_1} \mathds{1} _{x_1 \neq y_1} \,\pi_1(\mathrm dx_1, \mathrm dy_1)
	+
	\inf_{\pi_2}
	\iint_{\mathcal X_2\times \mathcal X_2} \mathds{1} _{x_2 \neq y_2} \,\pi_2(\mathrm dx_2, \mathrm dy_2),
\end{align*}
where the measures~$\pi_1$, $\pi_2$ on $\mc X_1 \times \mc X_1$ and $\mc X_2\times \mc X_2$, respectively, in the infima have marginals~$\prob_1, \probalt_1$ and $\prob_2, \probalt_2$, respectively.
Moreover, the lower bound follows by
\begin{align*}
	\inf_{\pi} \iint_{\mathcal X\times \mathcal X} \mathds{1} _{x \neq y} \,\pi(\mathrm dx,\mathrm dy)
	& =
	\inf_{\pi}
	\iint_{\mathcal X\times \mathcal X} \left( \max_{i=1,2} \mathds{1} _{x_i \neq y_i} \right)
	\,\pi(\mathrm dx, \mathrm dy)
	\geq \max_{i=1,2} 
	\, \inf_{\pi} \iint_{\mathcal X\times \mathcal X} \mathds{1} _{x_i \neq y_i} 
	\,\pi(\mathrm dx, \mathrm dy)\\[.3em]
	& =
	\max_{i=1,2} \inf_{\pi_i} \iint_{\mathcal X_i\times \mathcal X_i} \mathds{1} _{x_i \neq y_i} 
	\,\pi_i(\mathrm dx_i, \mathrm dy_i)
\end{align*}
with $\pi_1$, $\pi_2$ as above.
\end{proof}

Thus, for PDE models~\eqref{eq:pde} with random data, Proposition~\ref{propo:robust_TV} ensures---assuming merely measurability and no further continuity properties of the solution operator $u = \mc S(a,f)$---\emph{Lipschitz continuity} with Lipschitz constant one in TV distance for the propagation of uncertainty, i.e., the mapping $\prob\mapsto\mc S_\ast\prob$ with $(a,f)\sim \prob \in \mc P(\mc X)$.
Applying Proposition~\ref{propo:robust_TV} to measurable quantities of interest $\phi\colon \mc U\to \mathbb R$ also immediately yields for measurable $\mc S$
\[
	d_{\mathrm{TV}}\big((\phi \circ \mathcal S)_*\prob, 
	                   (\phi \circ \mathcal S)_*\probalt\big)
	\leq 
	d_{\mathrm{TV}}(\prob,\probalt).
\]
Moreover, for the case of product measures $\prob = \prob_{a} \otimes \prob_f$ and $\probalt = \probalt_{a} \otimes \probalt_f$ we have
\[
	d_{\mathrm{TV}}\left(\mathcal{S}_{*}\prob, \mathcal{S}_{*}\probalt\right)
	\leq 
	d_{\mathrm{TV}}(\prob_{a},\probalt_{a}) + d_{\mathrm{TV}}(\prob_f,\probalt_f),
\]
and analogously for $d_{\mathrm{TV}}\big((\phi \circ \mathcal S)_*\prob, (\phi \circ \mathcal S)_*\probalt\big)$.

However, probability measures $\prob_a,\probalt_a,\prob_f,\probalt_f$ on infinite dimensional function spaces such as $L^\infty(D)$ and $L^2(D)$ (e.g., for the elliptic problem) tend to be mutually singular, resulting in a maximal TV distance of one.
This is undesirable for sensitivity analysis and we therefore next consider sensitivity in the Wasserstein distance, which does not rely on absolute continuity of the probability measures.

\subsection{Sensitivity in Wasserstein Distance}\label{sec:Wasserstein_Sens}
In order to define the Wasserstein distance for probability measures $\prob$, $\probalt \in \mc P(\mc X)$ on a metric space $(\mc X,d_{\mc X})$ we define the sets
\[
	\mc P_p(\mc X)\coloneqq
	\left\{
	\prob \in \mc P(\mcX)\colon \int_{\mcX} d_{\mc X}(x,x_0)^p \, \prob(\mathrm dx) < \infty \text{ for some }x_0 \in\mc X
	\right\},
	\qquad
	p \geq 1.
\]
The \emph{$p$-Wasserstein distance} of $\prob$, $\probalt \in \mc P_p(\mc X)$ is then given by
\[
	d_p(\prob,\probalt)
	:=
	\inf_{\pi \in \Pi(\prob,\probalt)} 
	\left(\iint_{\mathcal X\times \mathcal X}d_{\mc X}(x,y)^p\,\pi(\mathrm dx,\mathrm dy)\right)^{\nicefrac1p},
\]
where $\Pi(\prob,\probalt)$ denotes the set of all measures $\pi \in \mathcal P(\mcX \times \mcX)$ with marginals $\pi(A\times \mathcal X) = \prob(A)$ and $\pi(\mathcal X\times B) = \probalt(B)$ for all Borel sets $A, B\subseteq \mathcal X$.
Any such measure $\pi \in \mathcal P(\mcX \times \mcX)$ with these marginals is called a \emph{coupling} of $\prob$ and $\probalt$.
We recall a few basic properties of the Wasserstein distance:
\begin{itemize}
\item
If $(\mathcal X,d_{\mc X})$ is complete and separable, then there always exists an optimal coupling $\pi^* \in \Pi(\prob,\probalt)$ for which the infimum in the definition of $d_p(\prob, \probalt)$ is attained, see \cite[Chapter~4]{Villani2009}.
\item
If $(\mathcal X,d_{\mc X})$ is complete and separable, then so is $(\mc P_p(\mc X), d_p)$, $p\geq 1$, see \cite{Villani2009}, and a dense subset is given by the convex hull of $\{\delta_x \colon x \in \mc X\}$.
\item
By Jensen's inequality we have that $d_q(\prob,\probalt) \leq d_p(\prob,\probalt)$ for any $\prob,\probalt\in \mc P_p(\mc X)$ and $1\leq q \leq p$.
\item
$(\mathcal X,d_{\mc X})$ embeds isometrically in $(\mathcal P_p(\mathcal X),d_p)$ via the embedding $x\mapsto\delta_x$, since $d_p(\delta_x,\delta_y)=d_{\mc X}(x,y)$ for all $p\ge1$, with $\delta_x$ the Dirac measure concentrated at $x\in\mc X$.
\end{itemize}

The following discussion considers general solution operators $\map\colon \mathcal X \to \mathcal U$ between metric spaces $(\mathcal X,d_{\mcX})$ and $(\mathcal U,d_{\mc U})$, representing the data  and solution spaces for a PDE model~\eqref{eq:pde}.
After stating the basic result we then specialize to the elliptic problem~\eqref{eq:BVP} with random data.
However, we emphasize that Theorems~\ref{theo:Lip} and~\ref{theo:LocLip} below hold for \emph{general forward maps} between metric spaces and are not necessarily related to the solution operators for PDEs.
We begin the discussion with globally H\"older continuous mappings~$\map$.


\begin{theorem}[Hölder continuity in Wasserstein distance] \label{theo:Lip}
Let $\prob$, $\probalt \in \mathcal P_p(\mathcal X)$, $p\geq 1$, and let $\map\colon \mathcal X \to \mathcal U$ satisfy Assumption~\ref{assum:Lip}.
%
Then for the pushforward measures $\map_*\prob$, $\map_*\probalt \in \mathcal P(\mc U)$ of $\prob$, $\probalt$ we have 
	\[
		d_p(\map_*\prob, \map_*\probalt) \leq C_{\map}\ d_{\beta p}(\prob,\probalt)^{\beta}
		\leq 
		C_{\map}\ d_{p}(\prob,\probalt)^{\beta}
	\]
with $d_p$ denoting the $p$-Wasserstein distance on $\mathcal P_p(\mcX)$ and $\mathcal P_p(\mc U)$, respectively.
In particular, we have $\map_*\prob$, $\map_*\probalt \in \mathcal P_{p}(\mc U)$.
\end{theorem}
\begin{proof}
Let $\pi \in \Pi(\prob,\probalt)$, then $\pi^{\map}\big(A\times B):=\pi\big(\map^{-1}(A)\times \map^{-1}(B)\big)$ satisfies $\pi^{\map} \in \Pi(\map_*\prob,\map_*\probalt)$ and we obtain by a change of variables
\begin{align*}
	\iint_{\mc U \times \mc U} d_{\mc U}(u_1,u_2)^p\pi^{\map}(\mathrm du_1,\mathrm du_2)
	&=
	\iint_{\mcX\times \mcX} d_{\mc U} \big(\map(x_1),\map(x_2)\big)^p\pi(\mathrm dx_1,\mathrm dx_2)\\
	&\le 
	C_{\map}^p\iint_{\mcX\times \mcX} d_{\mcX}(x_1,x_2)^{\beta p}\pi(\mathrm dx_1,\mathrm dx_2).
\end{align*}
The assertion follows by taking the infimum over all $\pi \in \Pi(\prob,\probalt)$ on both sides and noting that $\big\{\pi^{\map}\colon \pi \in \Pi(\prob,\probalt)\big\} \subseteq \Pi(\map_*\prob,\map_*\probalt)$.
\end{proof}

Theorem~\ref{theo:Lip} establishes that H\"older continuity of the forward map~$\map$ carries over to the pushforward mapping of measures $\mc P_{\beta p}(\mcX)\ni \prob \mapsto \map_*\prob \in \mc P_p(\mc U)$.

\paragraph{The elliptic problem~\eqref{eq:BVP}}
We now apply Theorem~\ref{theo:Lip} to the elliptic problem~\eqref{eq:BVP}.
To this end, we use the metric induced by $\| \cdot \|_{H_0^1(D)}$ on $\mc U = H_0^1(D)$ and 
\begin{equation}\label{eq:data_metric}
	d_{L_+^\infty(D)\times L^2(D)}\big( (a_1, f_1), (a_2, f_2) \big)
	:=
	\|\log a_1 - \log a_2\|_{L^\infty(D)} + \|f_1-f_2\|_{L^2(D)}.
\end{equation}
on the data space $\mc X := L_+^\infty(D)\times L^2(D)$.
Combining Theorem~\ref{theo:Lip} with Proposition~\ref{prop:Lipschitz} then yields the following result.

\begin{corollary}[Elliptic problem, measures with bounded support] \label{coro:BVP_Lip}
Let $r < \infty$ and $\prob$, $\probalt$ be probability measures for $a\in L_+^\infty(D)$ and $f \in L^2(D)$ with supports 
\begin{equation}\label{eq:bounded_support}
	\supp \prob,\ \supp \probalt \subseteq \{(a,f) \in L_+^\infty(D) \times L^2(D) \colon \|\log a\|_{L^\infty(D)}, \|f\|_{L^2(D)} \leq r\}.
\end{equation}
Then, the pushforward measures $\mathcal S_*\prob$, $\mathcal S_*\probalt \in \mathcal P_p\big(H_0^1(D)\big)$ of the random solutions to~\eqref{eq:BVP} satisfy
\begin{equation}\label{eq:Lip_estimate}
	d_p(\mathcal S_* \prob, \mathcal S_* \probalt) 
	\leq 
	c(1+r)\e^{3r} \ d_p(\prob, \probalt).
\end{equation}
\end{corollary}

We remark that the statement of Corollary~\ref{coro:BVP_Lip} also holds when using the metric
\begin{equation}\label{eq:data_metric_2}
	d_{L_+^\infty(D)\times L^2(D)}\big( (a_1, f_1), (a_2, f_2) \big)
	:=
	\|a_1 - a_2\|_{L^\infty(D)} + \|f_1-f_2\|_{L^2(D)}
\end{equation}
on $\mc X := L_+^\infty(D)\times L^2(D)$---given that~\eqref{eq:bounded_support}.
Then the constant $\e^{3r}$ in~\eqref{eq:Lip_estimate} can be replaced by $\e^{2r}$.

\subsection{Wasserstein Sensitivity for Locally Lipschitz Forward Maps}\label{sec:Robust_LocalLipschitz}

Theorem~\ref{theo:Lip} implies that the global Lipschitz constant of a pushforward map carries over to the mapping of the probability measures.
This statement relates to the Kantorovich--Rubinstein duality theorem (cf.\ Villani \cite{Villani2009}).
For forward maps which are only \emph{locally} Lipschitz continuous we can, in general, not expect global Lipschitz continuity for the pushfoward measures in the $p$-Wasserstein distance.
The following example  this issue and points out a particular situation where one can conclude at least local Lipschitz continuity for the probability measures for a forward map which is not globally Lipschitz.

\begin{example}[Locally versus globally Lipschitz forward maps]
Consider the Gaussian measures $\prob = \mathcal N(0,1)$ and $\probalt_{m} = \mathcal N(m,1)$ on the real line with mean $m\in\mathbb R$ and unit variance.
The $2$-Wasserstein distance of these Gaussian measures is (cf.\ 
Dowson and Landau
\cite{DowsonLandau1982})
\[
	d_{2}(\prob,\probalt_{m}) = |m|.
\]
The exponential function $\map(x):=\exp(x)$ on $\mc X = \mathbb R$ is locally, but not globally Lipschitz continuous.
By employing the dual representation \cite[Chapter~5]{Villani2003}
\begin{equation}\label{eq:KR}
	d_1(\prob,\probalt)
	=
	\sup_{\phi\colon \Lip(\phi)\leq 1} 
	\left| 
	\int_{\mc X} \phi(x) \ \prob(\mathrm dx) - \int_{\mc X} \phi(x) \ \probalt(\mathrm dx)
	\right|
\end{equation}
(the supremum is taken over all Lipschitz continuous functions $\phi\colon \mc X \to \mathbb R$ with Lipschitz constant $\Lip(\phi) \leq 1$) we obtain with $\phi(x)=x$ (and the formula for the mean of a lognormal distribution) that
	\[
		\sqrt{\e}|1-\exp(m)| 
		\leq 
		d_{1}(\map_*\prob, \map_*\probalt_{m}) 
		\leq 
		d_{2}(\map_*\prob, \map_*\probalt_{m}).
	\]
Thus, there is no constant $C<\infty$ which is independent of $m > 0$ such that
\[
	\sqrt{\e}|1-\exp(m)|
	\leq
	d_{2}(\map_{*}\prob, \map_{*}\probalt_{m})
	\le 
	C \ d_{2}(\prob,\probalt_{m})
	= 
	C |m|
\]
since $\frac{|1-\exp(m)|}{|m|}$ is unbounded.
However, restricting $|m| \leq M$ we have for $X \sim  \mathcal N(0,1)$ that 
\begin{align*}	
	d_{2}(\map_{*}\prob, \map_{*}\probalt_{m})
	&\leq 
	\ev{ |\map(X)-\map(m+X)|^2}^{1/2} 
	= 
	|1-\exp(m)| \ \ev{\exp(2X)}^{1/2}\\
	&= 
	\e |1-\exp(m)| 
	\leq 
	C_M\ d_{2}(\prob,\probalt_{m}),
\end{align*}
where $C_M:= \frac{ \e}M \exp(M)<\infty$.
The example demonstrates  that local Lipschitz forwards can yield (at best) local Lipschitz continuity in Wasserstein distance for the pushforward measures.
\end{example}

By employing the Cauchy--Schwarz inequality we now derive a Hölder bound in  $p$-Wasserstein distance of the pushforwards by the $2p$-Wasserstein distance of the input measures for local Hölder forwards.
We state the result in a more general form for locally H\"older continuous mappings, thus obtaining a H\"older bound in the corresponding Wasserstein distances.

\begin{theorem}[Local Hölder continuity in Wasserstein distance]\label{theo:LocLip}
Let $\map\colon \mcX \to \mc U$ satisfy Assumption~\ref{assum:LocLip}.
Furthermore, let $p \in [1,\infty)$ and $\prob$, $\probalt \in \mc P_{2\beta p}(\mc X)$ satisfy
\begin{equation}\label{eq:Lip_cond_meas}
    \int_{\mcX} C_{\map}\big(d_{\mc X}(x,x_0)\big)^{2p} \ \prob   (\mathrm dx) \le C 
	\quad\text{ and } \quad
	\int_{\mcX} C_{\map}\big(d_{\mc X}(x,x_0)\big)^{2p} \ \probalt(\mathrm dx) \leq C
\end{equation}
for a constant $C < \infty$.
Then the pushforward measures $\map_*\prob$, $\map_*\probalt \in \mathcal P_p(\mc U)$ satisfy
\[	
	d_p(\map_*\prob, \map_*\probalt) 
	\leq 
	2C^{\nicefrac 1{(2p)}} \cdot d_{2\beta p}(\prob,\probalt)^\beta.
\]
\end{theorem}
\begin{proof}
Employing the setting as in the proof of Theorem~\ref{theo:Lip} and local Lipschitz continuity, we have that 
\begin{align*}
	\evm{\pi^{\map}}{d_{\mc U}^p}
	&=
	\iint_{\mc U\times \mc U} d_{\mc U}(u_1,u_2)^p\pi^{\map}(\mathrm du_1,\mathrm du_2)
	= 
	\iint_{\mcX\times \mcX} 
	d_{\mc U}\big(\map(x_1),\map(x_2)\big)^p\pi(\mathrm dx_1,\mathrm dx_2)\\
	&\le 
	\iint_{\mcX\times \mcX} C_{\map}\big( d_{\mc X}(x_1,x_0) \lor d_{\mc X}(x_2,x_0)\big)^p d_{\mcX}(x_1,x_2)^{\beta p}\pi(\mathrm dx_1,\mathrm dx_2),
\end{align*}
where $a \lor b := \max(a,b)$. 
Applying the Cauchy--Schwarz inequality reveals that
\begin{align*}
	\evm{\pi^{\map}}{d_{\mc U}^p}
	&\le 
	\left(\int_{\mcX^2} C_{\map}\big( d_{\mc X}(x_1,x_0) \lor d_{\mc X}(x_2,x_0)\big)^{2p}  \pi(\mathrm dx_1, \mathrm dx_2)\right)^{\nicefrac12}
	\left(\int_{\mcX^2} d_{\mc X}(x_1,x_2)^{2\beta p}\pi(\mathrm dx_1, \mathrm dx_2)\right)^{\nicefrac12}
\end{align*}
and, since $C_{\map}(a \lor b)^{2p}\le\big( C_{\map}(a)+C_{\map}(b)\big)^{2p}\le 2^{2p-1}\big(C_{\map}(a)^{2p}+C_{\map}(b)^{2p}\big)$, we obtain 
\begin{align*}
	\evm{\pi^{\map}}{d_{\mc U}^p}
	&\le 
	\left(2^{2p}C\right)^{\nicefrac 12} 
	\iint_{\mcX \times \mcX} d_{\mcX}(x_1,x_2)^{\beta p}\pi(\mathrm dx_1,\mathrm dx_2),
\end{align*}
where we have used the fact that the marginals of $\pi$ are $\prob$ and $\probalt$.
Again, the assertion follows by taking the infimum over all $\pi \in \Pi(\prob,\probalt)$ on both sides and noting that $\left\{\pi^{\map}\colon \pi \in \Pi(\prob,\probalt)\right\} \subseteq \Pi(\map_*\prob,\map_*\probalt)$.
\end{proof}

Theorem~\ref{theo:LocLip} considers a local H\"older constant $C_{\map}(\cdot)$, which is integrable with respect to~$\prob$ and $\probalt$ as detailed in~\eqref{eq:Lip_cond_meas}. 
We then also obtain local H\"older continuity for the pushforward measures $\map_*\prob$ and $\map_*\probalt$ from $\mc P_p$ to $\mc P_{2p}$.
That is, the tails of $\prob$ and $\probalt$ decay faster than the local H\"older constant $C_{\map}$ grows.
We remark on two generalizations of this proposition before applying it to the locally Lipschitz solution operator $\mc S$ of the elliptic problem~\eqref{eq:BVP}.

\begin{remark}\label{rem:Hoelder}
The statement of Theorem~\ref{theo:LocLip} can be generalized by applying H\"older's inequality instead of the Cauchy--Schwarz inequality in the proof.
This allows to consider probability measures~$\prob$ and~$\probalt$ on~$\mcX$ which satisfy, for an arbitrary $q>1$ and $x_0\in\mcX$,
\[
	\int_{\mcX} C_{\map} \big(d_{\mc X}(x,x_0)\big)^{q} \ \prob(\mathrm dx),
	\int_{\mcX} C_{\map} \big(d_{\mcX}(x,x_0)\big)^{q} \ \probalt(\mathrm dx) 
	\le 
	C_q 
	< 	
	\infty
\]
and which belong to $\mc P_{\beta p \frac{q}{q-p}}(\mc X)$.
We then obtain for $p>q$
\[	
	d_p(\map_*\prob, \map_*\probalt) 
	\leq 
	(2C_q^{\nicefrac 1q})^p \cdot d_{\beta p \frac{q}{q-p}}(\prob,\probalt)^\beta.
\]
This generalization of Theorem~\ref{theo:LocLip} can be used in two ways: (i)~in order to relax the conditions on~$\prob$,~$\probalt$ or (ii)~in order to achieve $\frac{q}{q-p}$ close to $1$, i.e., obtaining an almost H\"older estimate in the $p$-Wasserstein distance. 
However, for the purposes of this paper, we will only work with Theorem~\ref{theo:LocLip} in what follows.
\end{remark}

\begin{remark}\label{rem:Product}
The main assumption of Theorem~\ref{theo:LocLip} can be refined for the case of a product space $\mc X = \mc X_1 \times \mc X_2$ equipped with the a metric $d_{\mcX}\big((x_1,x_2), (y_1,y_2)\big) \le d_1(x_1,y_1) + d_2(x_2,y_2)$, where $d_i$ denotes a metric on $\mc X_i$, $i=1,2$, by assuming that $\map\colon \mc X \to \mc U$ is locally H\"older continuous in the following way: there exists a nondecreasing $C_{\map} \colon [0,\infty) \times [0,\infty] \to [0,\infty)$ such that
\begin{equation}\label{eq:C_product}
		d_{\mc U}\big(\map(x_1,x_2),\map(y_1, y_2)\big) 
		\leq 
		C_{\map}(r_1, r_2)\,d_{\mcX}(x,y)^\beta,
\end{equation}
for all $x = (x_1,x_2)$ and $y=(y_1,y_2) \in \mc X$ with $x_i,y_i \in \mc X_i$ belonging to a ball with radius $r_i$ with respect to~$d_i$ around a center $z_i \in \mc X_i$.
We then obtain the same result as in Theorem~\ref{theo:LocLip} provided that
\begin{equation*}\label{eq:Lip_cond_meas_2}
	\int_{\mcX} C_{\map}\big(d_1(x_1,z_1), d_2(x_2,z_2)\big)^{2p} \ \prob   (\mathrm d x) 
	\le 
	C,
\end{equation*}
and analogously for $\probalt$.
We remark that local H\"older estimates of the form~\eqref{eq:C_product} are common for solution operators of PDE models~\eqref{eq:pde}, cf.\ Proposition~\ref{prop:Lipschitz} for the elliptic problem~\eqref{eq:BVP} stating that
\[
	\left\|\mathcal{S}(a_{2,}f_2)-\mathcal{S}(a_1,f_1)\right\|_{H_0^1(D)} 
	\leq
	 c(1+r_f)\e^{3r_a}\ \left(\,\|\log a_2 - \log a_1\|_{L^{\infty}(D)} + \|f_2-f_1\|_{L^2(D)}\right)
\]
for all $\|\log a_i\|_{L^\infty(D)}\leq r_a$ and $\|f_i\|_{L^2(D)}\leq r_f$.
\end{remark}

\paragraph{The elliptic problem~\eqref{eq:BVP} with lognormal diffusion coefficient}
Combining the preceding theorem with Proposition~\ref{prop:Lipschitz} and Remark~\ref{rem:Product} we obtain the following result.

\begin{corollary}[Elliptic problem, Wasserstein sensitivity] \label{coro:UQ_02}
	Let $\prob$, $\probalt\,$ be probability measures for  $(a,f)\in L_+^\infty(D) \times L^2(D)$ with
	\[
		\int (1+\|f\|_{L^2(D)})^{2p}\ \exp(6p\ \|\log a\|_{L^\infty(D)}) \, \mathrm d\prob
		\leq C < \infty
	\]	and \[
		\int (1+\|f\|_{L^2(D)})^{2p}\ \exp(6p \|\log a\|_{L^\infty(D)}) \, \mathrm d \probalt
		\leq C < \infty.
	\]
	Then the pushforward measures $\mathcal S_*\prob$, $\mathcal S_*\probalt \in \mathcal P\big(H_0^1(D)\big)$ of the resulting random solutions to~\eqref{eq:BVP} satisfy
	\[	d_p\big(\mathcal S_* \prob, \mathcal S_* \probalt\big)
		\leq 2c^2 C^{\nicefrac 1{(2p)}}\ d_{2p}(\prob, \probalt),
	\]
	where $c$ denotes the constant in~\eqref{eq:8} and the Wasserstein distance on $\mc P(L_+^\infty(D) \times L^2(D))$ is taken w.r.t.\@ the metric~\eqref{eq:data_metric}. 
\end{corollary}

We now discuss sufficient conditions on $\prob$, $\probalt$ such that the conditions of Corollary~\ref{coro:UQ_02} are satisfied.
To this end consider product measures $\prob = \prob_a \otimes \prob_f$ and $\probalt = \probalt_a \otimes \probalt_f$, where $\prob_a$, $\probalt_a \in \mc P(L_+^\infty(D))$ describe measures for the diffusion coefficient $a$ and $\prob_f$, $\probalt_f \in \mc P(L^2(D))$ describe measures for the source term~$f$ in~\eqref{eq:BVP}.
It is natural to require $\prob_f$, $\probalt_f \in \mc P_{2p}(L^2(D))$ in order to apply Corollary~\ref{coro:UQ_02}.
Concerning the measures $\prob_a$, $\probalt_a$ for the diffusion coefficient we consider the popular choice of lognormal random fields, i.e., $g:= \log a$ follows a Gaussian measure on $C(D)$ given by Gaussian random field model on $D$.
These models are characterized by a mean function $m \in C(D)$ and a continuous covariance function $c \in C(D\times D)$ which then describe the finite dimensional distribution of $(\log a(x_1),\ldots, \log a(x_n))$, $n\in\mathbb N$ and $x_i\in D$, by a Gaussian distribution $\mathcal N(\mathbf m, \mathbf C)$, where $\mathbf m = \big(m(x_1), \ldots, m(x_n)\big) \in \mathbb R^n$ and $\mathbf C \in \mathbb R^{n\times n}$ has entries $c_{ij} = c(x_i,x_j)$.
We denote the resulting Gaussian measures on $C(D)$ by $\mathcal N(m,c)$ and describe the resulting lognormal model for $a$ by $\log_*\prob = \mathcal N(m,c)$.
Since such lognormal models are a common choice for diffusion coefficients the question arises for which classes of mean functions $m$ and covariance functions $c$ we can ensure uniform integrability of $\exp(6p \|\log a\|_{L^\infty(D)})$, as required in Corollary~\ref{coro:UQ_02}.
In more detail: we seek suitable sets $\mc M \subseteq C(D)$ of mean functions and $\mc C \subseteq C(D\times D)$ of covariance functions for which there exists a finite constant $C< \infty$ with
\begin{equation}\label{eq:exp_mom}
	\int \exp(6p\ \|g\|_{L^\infty(D)})\ \prob(\mathrm d g) \leq C
	\quad
	\forall\prob \in \{\mathcal N(m,c)\colon m \in \mc M, c\in\mc C\}.
\end{equation}
\emph{Fernique's theorem} ensures finiteness of the above integral for a Gaussian measure $\mathcal N(m,c)$.
However, deriving a uniform bound $C<\infty$ for all Gaussians $\mathcal N(m,c)$ where $m$ and $c$ are allowed to vary within classes $\mc M$ and $\mc C$, respectively, is not trivial. 

By intuition $\mc M$, $\mc C$ have to be bounded, otherwise the mean and variance of $g(x)$ would be unbounded.
Here, we focus on a particular widely-used subclass of isotropic covariance functions (i.e., covariances of the form $c(\|x-y\|)$), the family of \emph{Mat\'ern covariance functions}.
This family with scale parameters~$\sigma>0$, $\rho>0$ and shape parameter $\nu>0$ is given by 
\[	c_{\sigma^2,\rho,\nu}(d)= \sigma^2 \frac{2^{1-\nu}}{\Gamma(\nu)}\left(\sqrt{2\nu}\frac d\rho\right) K_\nu\left(\sqrt{2\nu}\frac d\rho\right),\]
where $K_\nu$ is the Bessel function of the second kind.
For $\nu$ being half integer, i.e., $\nu = k+\frac 12$ for $k\in \mathbb N_0$, the corresponding covariance function simplifies to the explicit form
\begin{equation}\label{equ:Matern}
	c_{\sigma^2,\rho, k+\frac 12}(x,y)
	\coloneqq
	\sigma^2\frac{k!}{(2k)!} \sum_{i=0}^k \frac{(k+i)!}{i!(k-i)!} \left(2\frac{\sqrt{2k+1}}{\rho} |x-y| \right)^{k-i} \ \exp\left(- \frac{\sqrt{2k+1}}{\rho} |x-y| \right).
\end{equation}
For this class of covariance functions we have the following result, which is of interest also for other PDE settings.

\begin{theorem}[Uniform exponential moment bound]
\label{theo:lognormal}
Let $r_m < \infty$, $\sigma_{\max}<\infty$, $k_{\max} \in \mathbb N_0$ and $\rho_{\min} > 0$ be given and set
\[
	\mc M := \left\{m \in C(D) \colon	\|m\|_{C(D)} \leq r_m \right\},
	\quad
	\mc C = \left\{ c_{\sigma^2,\rho,k+\frac 12} : \sigma \leq \sigma_{\max}, \rho \geq \rho_{\min}, k \in \{0,\ldots,k_{\max}\} \right\}.
\]
Then, for any $p >0$ there exists a finite $C = C(p, r_m, \sigma_{\max}, k_{\max}, \rho_{\min}) < \infty$ such that~\eqref{eq:exp_mom} holds.
\end{theorem}
The proof of this theorem requires powerful tools from Gaussian process theory and is detailed in Appendix~\ref{app:GRF}.
As an immediate consequence of Theorem~\ref{theo:lognormal} and Corollary~\ref{coro:UQ_02} we obtain the following.

\begin{corollary}[Elliptic problem, lognormal diffusion] \label{coro:lognormal}
Let $\mathcal S\colon L^\infty_+(D)\times L^2(D) \to H_0^1(D)$ denote the solution operator of~\eqref{eq:BVP}.
Given $\mc M$ and $\mc C$ as in Theorem~\ref{theo:lognormal} and a radius $r_f < \infty$ there exists a constant $C<\infty$ such that 
\[	d_p\big(\mathcal S_* \prob, \mathcal S_* \probalt\big)
	\leq C\ d_{2p}(\prob, \probalt)
\]
for all product measures $\prob = \prob_a \otimes \prob_f$, $\probalt = \probalt_a \otimes \probalt_f$ on $L^\infty_+(D) \times L^2(D)$ with
\[
	\log_*\prob_a, \log_*\probalt_a \in \{\mathcal N(m,c)\colon m \in \mc M, c\in\mc C\}
\]
and
\[
	\evm{\prob_f}{\|f\|_{L^2(D)}^{2p}}, \; \evm{\probalt_f}{\|f\|_{L^2(D)}^{2p}} \leq r_f.
\]
Here the Wasserstein distance derives from the metric induced by $\|\cdot\|_{H_0^1(D)}$ and either~\eqref{eq:data_metric} or~\eqref{eq:data_metric_2} on $L^\infty_+(D)\times L^2(D)$.
\end{corollary}
\begin{proof}
We have to verify the assumptions of Corollary~\ref{coro:UQ_02} for $\prob$ and $\probalt$.
Notice that, due to the product structure, we have
\[
	\int (1+\|f\|_{L^2(D)})^{2p}\ \exp(6p\ \|\log a\|_{L^\infty(D)}) \mathrm d\prob
	=
	\int (1+\|f\|_{L^2(D)})^{2p}\ \mathrm d\prob_f
	\
	\int \e^{6p\|\log a\|_{C(D)}}\ \mathrm d\prob_a
\]
and analogously for $\probalt$.
The first term on the right is uniformly bounded by $2^{2p}(1+r_f)$ given the assumptions on $\prob_f$ and $\probalt_f$.
The boundedness of the second term follows by Theorem~\ref{theo:lognormal}.
\end{proof}

\begin{remark}
In view of Remark~\ref{rem:Hoelder} we can modify Corollary~\ref{coro:lognormal} in the following way:
let the source term be deterministic, i.e., $\prob_f = \probalt_f = \delta_{f_0}$ for an  $f_0 \in L^2(D)$, 
and let $\prob_a$, $\probalt_a$ be as in Corollary~\ref{coro:lognormal}. Then, for any $\varepsilon > 0$, there exists a constant $C = C(\varepsilon)<\infty$ such that for $\prob = \prob_a\otimes \delta_{f_0}$ and $\probalt = \probalt_a\otimes \delta_{f_0}$ we have
\[	
	d_p\big(\mathcal S_* \prob, \mathcal S_* \probalt\big)
	\leq C\ d_{p+\epsilon}(\prob, \probalt).
\]
Hence, we almost obtain Lipschitz continuity in the $p$-Wasserstein distance for lognormal models. However, the constant depends on $\varepsilon$ and $C(\varepsilon) \to \infty$ as $\varepsilon\to0$.
\end{remark}

\begin{remark}
The presented results for the elliptic problem, Corollaries~\ref{coro:BVP_Lip},~\ref{coro:UQ_02}, and~\ref{coro:lognormal}, hold for both metrics ~\eqref{eq:data_metric} and~\eqref{eq:data_metric_2} on $L^\infty_+(D)\times L^2(D)$.
The choice of metric slightly changes the constants appearing in the estimates, but more importantly, it determines the corresponding Wasserstein distance of $\prob_a$ and $\probalt_a$.
If $\prob_a, \probalt_a$ are probability measures associated with lognormal random field models as in Corollary~\ref{coro:lognormal}, then the metric~\eqref{eq:data_metric} seems more convenient, since then $d_{p}(\prob_a, \probalt_a)$ coincides with the Wasserstein distance of the associated Gaussian measures on $C(D)$ taken w.r.t.\@ the norm on $\|\cdot\|_{C(D)}$.
\end{remark}

\subsection{Sensitivity with respect to truncation} \label{sec:StrawberryFields}

A convenient representation of random fields on a bounded domain $D\subset\mathbb R^d$, such as $a$ or $f$ in the elliptic problem~\eqref{eq:BVP}, are series expansions with random coefficients, e.g.,
\begin{equation}\label{eq:RFexp}
	f(x, \omega) = f_0(x) + \sum_{k=0}^\infty\sigma_k\, \xi_k(\omega)\, f_k(x), 
	\qquad x\in D,
\end{equation}
where $\xi_{k}$ are mutually uncorrelated mean-zero real-valued random variables with unit variance; $\{f_k\}_{k\in\mathbb N}$ a suitable system of normalized basis functions and $f_0$ represents the mean function of the random field, i.e., $f_0(x) = \ev{f(x,\cdot)}$. 
The most common such expansion is the \emph{Karhunen--Lo\`eve} expansion (KLE): let $f$ be a random field with continuous mean $f_0$ and continuous covariance function $c(x,y) = \cov\big(f(x), f(y)\big)$ and let $C\colon L^2(D) \to L^2(D)$ denote its (trace-class) covariance operator in $L^2(D)$ given by $C\varphi(x) := \int_{D}c(x,y)\,\varphi(y)\,\mathrm dy$.
Then, the KLE of~$f$ is given by~\eqref{eq:RFexp}, where $(\sigma^2_k, f_k)$ are the eigenpairs of the operator~$C$. 
However, the eigensystem of the covariance operator is not the only suitable system for expanding random fields.
In general, any Parseval frame of $L^2(D)$ will yield a similar expansion with uncorrelated random coefficients \cite{LuschgyPages2009}.

For computational purposes one can then truncate an expansion~\eqref{eq:RFexp} after sufficiently many ($K\in\mathbb N$, say) terms and work with the truncated random field
\begin{equation}\label{eq:RFexp_trunc}
	f_K(x,\omega) := f_0(x) + \sum_{k=0}^{K} \sigma_{k}\,\xi_{k}(\omega)\,f_k(x)
\end{equation}
as uncertain coefficient in~\eqref{eq:BVP}.
In order to study the Wasserstein distance of the resulting distributions~$\prob$, $\prob_K\in \mc P(L^2(D))$ of $f$ and $f_K$, respectively, (taken w.r.t.\@ the metric induced by the $L^2(D)$-norm) we can use
\[
	d_p(\prob,\ \prob_K) 
	\leq 
	\ev{\|f - f_K\|^p_{L^2(D)}}^{1/p},
\]
since the distribution of $(f,f_K)$ is obviously a coupling of $\prob$ and $\prob_K$.
For $p=2$ the right-hand side is explicitly
\begin{equation}\label{eq:20}
	d_2(\prob,\ \prob_K) \leq \sqrt{\sum_{k=K+1}^\infty \sigma^2_k}.
\end{equation}
In case of Gaussian random fields $f$ and their truncated KLE we even obtain equality.
\begin{proposition}[$L^2$-Truncation error]  \label{theo:Trunc_Hilbert}
Let $f$ denote a Gaussian random field on $D$ with continuous mean and covariance and let $(\sigma_k^2,f_k)$ denote the eigenpairs of its covariance operator on $L^2(D)$.
Then for $f_K$ as in~\eqref{eq:RFexp_trunc} we have for the resulting distribution $\prob, \prob_K \in \mc P(L^2(D))$ of $f$ and $f_K$ that
\begin{equation}\label{eq:21}
	d_2(\prob,\ \prob_K) = \sqrt{\sum_{k=K+1}^\infty \sigma^2_k},
\end{equation}
where the Wasserstein distance is taken w.r.t.\@ metric induced by the $L^2(D)$-norm.
\end{proposition}
\begin{proof}
By construction $\prob = \mathcal N(f_0, C)$ and $\prob_K = \mathcal N(f_0, C_K)$ are Gaussian distributions on the Hilbert space $L^2(D)$ where $C_K \varphi := \int_D c_K(x,y) \varphi\, \mathrm d y$ with $c_K(x,y) := \cov(f_K(x), f_K(y)) = \sum_{k=1}^K\sigma_k^2\,f_k(x)\,f_k(y)$.
	For Gaussian measures on Hilbert spaces there exists an exact formula for their $2$-Wasserstein distance
Gelbrich
\cite{Gelbrich1990} which in this case is 
	\[
		d_2(\prob,\ \prob_K)^2
		=
		d_2\big(\mathcal N(f_0, C),\ \mathcal N(f_0, C_K)\big)^2
		=
		\tr(C) + \tr(C_K) - 2 \tr\left( \sqrt{C_K^{\nicefrac 12}\, C \, C_K^{\nicefrac 12}}\right).
	\]
	We have that $\tr(C) = \sum_{k=1}^\infty \sigma_k^2$ and $\tr(C_K) = \sum_{k=1}^K \sigma_k^2$.
	Moreover, $C_K$ and $C$ share the same eigensystem and the null space of $C_K$ is the closure of the span of $\{f_{k}\colon k >K\}$.
	Thus, $C_K^{\nicefrac 12}\, C\, C_K^{\nicefrac 12}$ has the eigenpairs $(\widetilde \sigma^2_k , f_k)$ with $\widetilde \sigma^2_k = \sigma^4_k$ for $k=1,\ldots,K$ and $\widetilde \sigma^2_k = 0$ for $k > K$.
	This leads to,
	\[
		\tr(C) + \tr(C_K) - 2 \tr\left( \sqrt{C_K^{\nicefrac 12}\, C\, C_K^{\nicefrac 12}}\right)
		=
		2\sum_{k=1}^K \sigma^2_k
		+ \sum_{k>K} \sigma^2_k
		-
		2
		\sum_{k=1}^K \sigma^2_k
		=
		\sum_{k>K} \sigma^2_k,
	\]
	as desired.
\end{proof}

A similar statement holds for the Wasserstein distance of truncated random fields in Banach space norms.
We consider here the case when the logarithm $\log a$ of a positive diffusion coefficient $a \in L^\infty_+(D)$ is expanded in a series analogous to~\eqref{eq:RFexp}.
The case of truncating suitable series expansions for $a$ itself can be treated similarly.
\begin{proposition}[$L^\infty$ truncation bound] \label{theo:Trunc_Banach}
	Let
	\begin{equation}\label{eq:KLE_a}
		\log a(x, \omega)
		=
		a_0(x) + 
		\sum_{k=1}^\infty \sigma_k \, \xi_k(\omega)\, a_k(x),
		\qquad
		x\in D,
	\end{equation}
	converge almost surely in $L^\infty(D)$ with uncorrelated mean-zero random variables $\xi_k$ where $\var[\xi_k]=1$.
	Let $\log a_K$ denote the random field resulting from truncating the series in~\eqref{eq:KLE_a} after $K$ terms and let $\prob, \prob_K \in \mc P(L^\infty_+(D))$ denote the distributions of $a$ and $a_K$, respectively.
	If there exists a constant $C<\infty$ such that $\ev{|\xi_k|} < C$ for all $k\in\mathbb N$, then
	we have for the Wasserstein distance taken w.r.t.\@ the metric $d(a,a_K) := \|\log a -\log a_K\|_{L^\infty(D)}$ on $L^\infty_+(D)$ that
	\begin{equation}
		d_1\left(\prob,\, \prob_K\right) \leq C \sum_{k=K+1}^\infty |\sigma_k| \|a_k\|_{L^\infty(D)}.
	\end{equation}
	In the case of bounded random coefficients, i.e., $|\xi_k| \leq C$ almost surely for all $k\in\mathbb N$, we also have for any $p\geq 1$
	\begin{equation}\label{eq:trunc_loga_2}
		d_p\left(\prob,\, \prob_K\right) \leq C \sum_{k=K+1}^\infty |\sigma_k| \|a_k\|_{L^\infty(D)}.
	\end{equation}
\end{proposition}
\begin{proof}
Let $\pi \in \Pi(\prob, \prob_K)$ denote the distribution of the pair $(\log a, \log a_K)$ on $L^\infty(D)\times L^\infty(D)$.
Then
\begin{align*}
	d_p\left(\prob,\ \prob_K\right)
	& \leq \ev{\|\log a - \log a_K\|^p_{L^\infty(D)}}^{1/p}
	\leq \ev{\left( \sum_{k>N} |\sigma_{k}|\,|\xi_k|\, \|a_k\|_{L^\infty(D)}\right)^p}^{1/p}.
\end{align*}
For the bounded case, $|\xi_k| \leq C$ almost surely for all $k\in\mathbb N$, we obtain the second statement immediately.
And for $p=1$ the first statement follows easily by taking the expectation inside the series in the above inequality.
\end{proof}
\begin{remark}\label{rem:Charrier}
A result of Charrier \cite{Charrier2012} on the truncation error in the $L^p(\Omega; L^\infty(D))$-norm can be used bound the Wasserstein distance in~\eqref{eq:trunc_loga_2}.
For KLE and under appropriate assumptions \cite{Charrier2012}, these yield
\[
	d_p\left(\prob,\ \prob_K\right)
	\leq
	C_p \max\left(\sum_{k=K+1}^\infty \sigma^2_k \|a_k\|^2_{L^\infty(D)}, \sum_{k=K+1}^\infty \sigma^2_k \|a_k\|^{2\alpha}_{L^\infty(D)}\|\nabla a_k\|^{2(1-\alpha)}_{L^\infty(D)}  \right)^{\nicefrac 12},
\]
where $\alpha\in(0,1)$ is such that the second series converges.
\end{remark}


\section{Risk Functionals for Uncertainty Analysis} \label{sec:Risk}

The focus of an uncertainty propagation analysis involving random PDEs is typically a quantity of interest associated with the PDE solution, from which useful information may be extracted by statistical post-processing. 
Risk functionals can be viewed as a specific type of post-processing, as they condense the probability distribution of a random variable into a number reflecting the impact of its fluctuations for a particular quantity of interest.
As such, they are crucial to linking an uncertainty propagation analysis to the underlying application.

\subsection{Risk Functionals} 

Risk functionals have gained increased relevance in the recent past, initially driven by mathematical finance and subsequently adopted by the engineering and optimization communities, as detailed in the Introduction.
Risk functionals assign real numbers to random variables in such a way that these values constitute a measure of the \emph{risk} associated with their random outcomes. 
As such they are defined on a vector space $\mc L$ of real-valued random variables. 
In the uncertainty propagation setting for random PDEs, a risk functional $\rho$ would typically be applied to the output quantities of interest, resulting in the chain of mappings
\[
    \underbrace{
    {\mc X} \xrightarrow{\; {\mc S} \;} 
    {\mc U} \xrightarrow{\; \phi \;} 
    \mathbb R
	}_{\rho\colon \mc L\to\mathbb R} 
%
\]
in which the sequence above the underbrace denotes the deterministic problem and randomness enters below due once a probability distribution is introduced on $\mc X$.
\begin{definition}[Risk functional, cf.\ Artzner et al.\ \cite{Artzner1997}] \label{def:RiskMeasure}
	Assume an abstract probability space $(\Omega, \mc F, \prob)$ and a vector space $\mathcal L$ of real-valued random variables $X\colon \Omega \to \mathbb R$.
	A mapping $\Risk\colon\mathcal L\to\mathbb{R}\cup\{\infty\}$ is a \emph{risk functional} if it satisfies the following axiomatic properties for $X$, $Y \in\mathcal L$.
	\begin{subequations}
		\begin{align}
		& \Risk(X) \le \Risk(Y), \text{ if } X\le Y \text{ a.s.} 
			&& \text{(monotonicity)}, 							\label{enu:Eins}\\
		& \Risk(X+c) = \Risk(X)+c \text{ for } c\in\mathbb R    
			&& \text{(translation equivariance)}, 				\label{enu:Zwei} \\
		& \Risk(X+Y) \le \Risk(X) + \Risk(Y) 
			&&\text{(subaddititvity)}, 							\label{enu:Drei}\\
		& \Risk(\lambda\cdot X) = \lambda\cdot \Risk(X) \text{ for } \lambda>0 
			&&\text{(positive homogeneity)}. 					\label{enu:Vier} 
		\end{align}
	\end{subequations}
\end{definition}

These  properties  possess natural interpretations in the context of risk management and insurance.
In a financial context risk functionals are employed to price insurance policies by mapping the policy to the premium. 
Here subadditivity~\eqref{enu:Drei} expresses that combining risk in a single policy is preferable to issuing separate insurance contracts.
In an engineering setting, positive homogeneity~\eqref{enu:Vier} implies that rescaling by a change of physical units leaves risk unaffected.
Occasionally in the literature the term \emph{risk functional} can also be found for mappings $\Risk$ satisfying only~\eqref{enu:Eins}--\eqref{enu:Drei}, while risk functionals satisfying also~\eqref{enu:Vier} are then referred to as \emph{coherent} risk functionals.

\begin{remark}[Domain of risk functionals]
	Ruszczyński and Shapiro	\cite{Ruszczynski2006} discuss risk functionals on $L^p$ space ($p\in [1,\infty]$).
	They conclude that $\Risk$ is either continuous on $L^p$ or the set $\{X\in L^p\colon \Risk(X)=\infty\}$ is dense in $L^p$.
	To specify the largest class of random variables with $\|\cdot\|_\rho <\infty$ one may associate a norm and a domain $\mathcal L$ with a risk functional $\Risk$ in a natural way.
	To this end define
	\begin{equation}\label{eq:Norm2}
		\left\Vert X \right\Vert_{\Risk} := \Risk(|X|) 
		\quad \text{ for } \quad X \in 
		\mathcal L \coloneqq \left\{ 
			X\colon\Omega \to \mathbb R  \text{ measurable with } \left\Vert X\right\Vert_{\Risk} <\infty
			\right\}.
	\end{equation}	
	The pair $\bigl(\mathcal L,\left\Vert \cdot\right\Vert_{\Risk}\bigr)$ is then a Banach space, the largest possible for which $\Risk$ is finite on $\mathcal L$, cf.\ Pichler \cite{Pichler2013a}. The space $\mathcal L$ is most typically a Lorentz rearrangement space, cf.\ Pichler and Kalmes \cite{KalmesPichler}.
\end{remark}

\begin{remark}[Representation by convex duality] \label{rem:Dual}
	In what follows let~$\mathcal L^*$ denote the dual (or pre-dual) of $\mathcal L$ with duality pairing
	\[	\langle X,Z\rangle
	= \ev{X\,Z} = \int_\Omega X(\omega)\,Z(\omega) \,\prob(\mathrm d \omega),
	\qquad  X \in\mathcal L,\, Z \in \mathcal L^\ast.\] 
	For any subset 
	\begin{equation} \label{eq:support}
		\mathcal A\subset\left\{ Z\in\mathcal L^*\colon \ev{Z} = 1 \text{ and } Z\ge0\right\},
	\end{equation}
	a risk functional satisfying all properties~\eqref{enu:Eins}--\eqref{enu:Vier} above is given by
	\begin{equation}\label{eq:6}
		\Risk(X):=\sup\left\{ \ev{X\,Z}\colon Z\in\mathcal{A}\right\}.
	\end{equation}
	We then call $\mathcal A$ the \emph{support set} of $\Risk$ in~\eqref{eq:6}.

	By~\eqref{enu:Drei} and~\eqref{enu:Vier}, any risk functional $\Risk(\cdot)$ is convex.
	It follows from convex duality (the Fenchel--Moreau theorem, cf.\ 
	Rockafellar
	\cite{Rockafellar1974}) that
	\[ 
		\Risk(X) = \sup\big\{\ev{X\,Z} - \Risk^*(Z) : Z\in\mathcal L^*\big\},
	\]
	where for $Z \in\mathcal L^\ast$
	\[ 
		\Risk^*(Z):= \sup\big\{\ev{X\,Z} - \Risk(X) : X\in\mathcal L\big\}
	\]
	is the convex dual function to $\Risk$ (cf.\ 
	Shapiro et al.\ 
	\cite{RuszczynskiShapiro2009}). 
	It follows from~\eqref{enu:Zwei} and~\eqref{enu:Eins} that
	\[
		\Risk^*(Z)
		=
		\begin{cases}
			0       & \text{ if }\,\ev{Z}=1 \text{ and } Z\ge 0,\\
			+\infty & \text{ otherwise}.
		\end{cases}
	\]
	We may thus define the support set of any risk functional $\Risk$ by 
	\begin{equation} \label{eq:support-dual}
		\mathcal A
		:=
		\big\{Z\in\mathcal L^*\colon \Risk^*(Z)<\infty\big\},
	\end{equation}
	so that~\eqref{eq:6} applies.
		
	The random variables $Z\in\mathcal A$ are \emph{densities} with respect to the probability measure $\prob$: they are nonnegative and they satisfy $\ev{Z}=1$.
	The support set defined in~\eqref{eq:support-dual} thus consists of densities (cf.~\eqref{eq:support}).
    Moreover, the set $\mathcal{A}$ in~\eqref{eq:support-dual} is weak{*} closed, and hence the supremum in~\eqref{eq:6} is attained. 
	\end{remark}
	\begin{remark}[Assessment and quantification of risk]
	The random variables $Z\in \mathcal A$ provide a useful interpretation of the risk functional. 
	To this end suppose that $Z^*$ is optimal in~\eqref{eq:6} so that 
	\begin{equation}\label{eq:weight}
		\Risk(X) = \sup_{Z\in\mathcal A}\ev{X\,Z} = \ev{X\,Z^*}.
	\end{equation}
	Then $Z^*$ acts as a weight, as it is nonnegative and integrates to~$1$ ($\ev{Z}=1$).
	The weighted expectation $\ev{X\,Z^*}$ in~\eqref{eq:weight} weights every outcome $X(\omega)$ with $Z^*(\omega)$ in the worst possible way ($Z^*$ attains the supremum):
	unfavorable outcomes $X(\omega)$ will be overvalued and assigned a high weight $Z^*(\omega)>1$, while favorable outcomes $X(\omega)$ will be assigned a lower weight $Z^*(\omega)\le1$.
	In this interpretation, the random variable $Z^*$ is an individual assessment of risk for the particular random variable $X$.
\end{remark}

\begin{example}[Risk neutrality and maximal risk aversion] \label{ex:1}
The risk functional 
\[
	\Risk(X) := \ev{X}
\] 
is the simplest functional satisfying all axioms above, it is called the \emph{risk neutral} risk functional, as it ignores fluctuations around the mean completely.
Its support set $\mathcal A=\left\{ \one\right\}$ consists of a single element, the constant density $\one(\cdot) \equiv1$.
	
By contrast, the functional 
\[	\Risk(X) := \esssup X	\] 
is the most conservative risk functional.
It indicates maximal risk aversion, as it represents the risk associated with the random outcome $X$ by its largest possible outcome. 
It has the maximal support set $\mathcal A=\{Z\colon Z\ge 0 \text{ and }\ev{Z}=1\}$.
\end{example}

\begin{example}[Average value-at-risk]
The \emph{average value-at-risk} is the most prominent example of a risk functional. 
At risk level $\alpha\in[0,1)$, this functional is given by
\begin{equation}\label{eq:19}
	\AVaR_\alpha(X)\coloneqq\frac1{1-\alpha}\int_\alpha^1 F_X^{-1}(u)\,\dd u,
\end{equation}
where 
\begin{equation}\label{eq:22}
	F_X^{-1}(u) \coloneqq \VaR_u(X) \coloneqq \inf\big\{x\in\mathbb R \colon \prob(X \le x)\ge u \big\}
\end{equation}
is the \emph{quantile function}, or \emph{value-at-risk} at level $\alpha$.\
Note that the value-at-risk itself is not a risk functional in the sense of Definition~\ref{def:RiskMeasure}, since it does not satisfy subadditivity~\eqref{enu:Drei}, but~\eqref{eq:19} constitutes the convex envelope of~\eqref{eq:22}, cf.\ \cite{Follmer2004}.
The support set 
\begin{equation}
	\mathcal{A}	=
	\left\{ Z\colon \ev{Z}=1\text{ and }0\le Z\le\frac1{1-\alpha}\text{ a.s.}\right\}
\end{equation}
allows a representation of $\AVaR$ as a supremum as in~\eqref{eq:6}. 
By contrast, the representation
\begin{equation}\label{eq:AVaR}
	\AVaR_{\alpha}(X)
	:=
	\inf_{\ q\in\mathbb{R}}  q+\frac1{1-\alpha}\ev{(X-q)_+},
	\quad 
	\text{ where } x_+:=\max(0,x),
\end{equation}
as an infimum derives from convex duality, cf.\ 
Ogryczak and Ruszczyński
\cite{RuszOgryczak}, 
Rockafellar and Uryasev
\cite{RockafellarUryasev2000} and 
Pflug
\cite{Pflug2000}.
The average value-at-risk is also known as \emph{conditional value-at-risk}. 
Actuaries, however, prefer the terms \emph{expected shortfall} or \emph{conditional tail expectation}.
\end{example}
\begin{example}[Semideviation]
The semideviation risk functional selectively penalizes deviations above the mean and is defined by 
\[
	\Risk(X):= \ev{X} + \beta\cdot\ev{\big(X-\ev{X}\big)_+},
\]
where $\beta\in[0,1]$ is a coefficient of risk aversion.
More generally, for $p \ge 1$ the $p$-semideviation is
\[
	\Risk(X):=\ev{X} + \beta\cdot\big\|(X-\ev{X})_+\big\|_p.
\]
The semideviation has the alternative representation in terms of average value-at-risk
\[	\Risk(X) =	\sup_{\kappa\in (0,1)} (1-\beta\,\kappa)\ev{X} + \beta\,\kappa\,\AVaR_{1-\kappa}(X).\]
A slightly more complicated version is available for the $p$-semideviation as well and given in 
Pichler and Shapiro
\cite[Corollary~6.1]{ShapiroAlois}.
\end{example}

\begin{example}[Spectral risk functionals]
	\emph{Spectral risk functionals} constitute a class of risk functionals which reweight the quantiles~\eqref{eq:22}.
	They are defined in terms of a \emph{spectral function} $\sigma\colon[0,1)\to\mathbb R_0^+$ satisfying $\int_0^1\sigma(u)\,\mathrm du=1$ by
	\begin{equation}\label{eq:spectral}
		\Risk_{\sigma}(X) := \int_0^1 \sigma(u)\,F_X^{-1}(u)\,\mathrm du,
	\end{equation}
	are another way of quantifying risk, which for
	$\sigma = \frac{1}{1-\alpha}\mathds{1}_{[\alpha,1]}$
	recovers average value-at-risk.
	Here the support set is given by (cf.\ 
	Pichler
	\cite{Pichler2013b})
	\begin{equation}\label{eq:Spectral}
		\mathcal{A}_\sigma
		\coloneqq
		\left\{ 
		Z\colon Z\ge0,\ \AVaR_{\alpha}(Z)\le\frac1{1-\alpha}\int_\alpha^1\sigma(u) \,\dd u \text{ for all }\alpha<1
		\right\}.
	\end{equation}
\end{example}

\begin{example}[Entropic value-at-risk]\label{ex:2}
	The \emph{entropic value-at-risk} is a logarithmic analogue of the average value-at-risk. Partucularly the representation~\eqref{eq:AVaR} gives rise to consider
	\begin{equation}\label{eq:EVaR}
		\EVaR_{\alpha}(X):= \inf_{t>0}\ \frac1{t} \log\frac1{1-\alpha} \ev{e^{tX}}.
	\end{equation}
	Its support set is
	\[	\mathcal{A}
		\coloneqq \left\{ Z\ge0\colon \ev{Z}=1\text{ and }\ev{ Z\log Z} \le \log\frac1{1-\alpha}\right\},\]
	where $H(Z):=\ev{Z\log Z}$ is the entropy of the density~\(Z\).
\end{example}

The risk functionals exemplified above concentrate the outcomes of a random variable in a single real number. In this way it is possible to assess outcomes, which are not known beforehand.

The following section addresses the question of how to valuate these outcomes, as the precise distribution is typically not known. For this reason we provide the corresponding sensitivity analysis in the next section, which finally gives full access to assessing the random outcome in the uncertain environment.

\subsection{Sensitivity Analysis for Risk Functionals}  \label{sec:Bivariate}

All risk functionals discussed above depend in different ways on the underlying probability measure $\prob$ of the probability space $(\Omega,\mc F, \prob)$. In an uncertain setting, however, $\prob$ is not known precisely.
It may be approximated by statistical estimation of parameters defining a family of probability distributions; alternatively, only an empirical measure 
$\prob_n:=\frac1n\sum_{i=1}^n \delta_{X_i}$
may be available.
To discuss uncertainty it is thus essential to account for variations of the probability measure $\prob$ and investigate the impact resulting from such imprecise knowledge.

We now study the sensitivity of risk functionals with respect to the underlying probability measure~$\prob$, measuring perturbations of the latter in Wasserstein distance.
To this end, consider $\prob$, $\probalt \in \mc P(\mc X)$. Evaluations with respect to different probability measures are made explicit by writing the probability measure as a subscript, e.g., $\evm{\prob}{X} = \int_{\mc X} X\,\mathrm d\prob$ and $\evm{\probalt}{X}=\int_{\mc X} X\,\mathrm d\probalt$.
In order to analyze the effect of the probability measure $\prob$ on the value $\Risk(X)$ 
we consider an associated risk functional $\Risk_\pi$ for random variables $X'\colon \mc X\times \mc X \to \mathbb R$ on the product space  equipped with a coupling $\pi \in \Pi(\prob, \probalt)$ as the underlying probability measure.
That is, we consider 
\[	
	\Risk_\pi(X^\prime) 
	:= 
	\sup\left\{ \evm{\pi}{X^\prime\,Z}\colon Z\in\mathcal{A_\pi}\right\},
\]
where
\[	
	\mathcal{A_\pi}\subset\left\{ Z\colon \mc X \times \mc X\to [0,\infty) 
	\text{ such that }
	\evm{\pi}{Z}=1\right\}
\]
denotes the support set of $\Risk_\pi$ on the product space.
Furthermore, let $\mathtt p_i\colon \mc X\times \mc X \to \mc X$, $i=1,2$, with $\mathtt p_i\big((x_1,x_2)\big):= x_i$ denote the canonical projections so that $\prob = \pi\circ \mathtt p^{-1}_1$ and $\probalt = \pi\circ\mathtt p^{-1}_2$.
We thus may define $\Risk(X)$ for a $X\colon \mc X\to \mathbb R$ given either $\prob$ or $\probalt$, by
\begin{equation}\label{eq:RP}
	\Risk_\prob(X) 
	:=	
	\Risk_\pi(X \circ \mathtt p_1)
	=
	\sup\left\{ \evm{\pi}{Z\cdot (X \circ \mathtt p_1)}\colon Z\in\mathcal{A}_\pi\right\}
\end{equation}
and
\begin{equation}\label{eq:RQ}
	\Risk_\probalt(X)
	:=	
	\Risk_\pi(X\circ \mathtt p_2)
	=
	\sup\left\{ \evm{\pi}{Z\cdot (X\circ \mathtt p_2)}\colon Z\in\mathcal{A}_\pi\right\}.
\end{equation}
\begin{remark}
The converse is possible as well. 
Appendix~\ref{app:Risk} explicitly constructs a risk functional $\Risk_\pi$ from its marginals $\Risk_\prob$ and $\Risk_\probalt$ as in~\eqref{eq:RP} and~\eqref{eq:RQ}. 
\end{remark}

\begin{remark}
\emph{Law-invariant} risk functionals $\Risk$ depend on the cumulative distribution function of $X$ only (see~\eqref{eq:spectral}, e.g.). For these risk functionals it is obvious that~\eqref{eq:RP} and~\eqref{eq:RQ} are consistent definitions:
consider, for instance, the average value-at-risk and a coupling $\pi\in \Pi(\prob,\probalt)$ for two arbitrary probability measures on $\mc X$.
Then, for any random variable $X\colon \mc X \to \mathbb R$, we have
\begin{align*}
	\AVaR_{\alpha,\pi}(X\circ \mathtt p_1) 
	& = \inf_{q\in\mathbb{R}}q+\frac1{1-\alpha}\evm{\pi}{\big(X\circ \mathtt p_1-q\big)_+}\\
	& = \inf_{q\in\mathbb{R}}q+\frac1{1-\alpha}\evm{\prob}{\big(X-q\big)_+}
	= \AVaR_{\alpha,\prob}(X) 
\end{align*}
and analogously that $\AVaR_{\alpha,\pi}(X\circ \mathtt p_2)=\AVaR_{\alpha,\probalt}(X)$.
	
We note that all examples in the preceding subsection are law-invariant risk functionals.
However, the setting outlined here also allows the analysis of risk functional which are not law-invariant.
Such risk functionals appear, for example, in insurance when the cause of the loss $X(\omega)$ is also important: damage caused by floods or natural disasters can pose a higher risk to insurance companies (due to cross-correlations) than damage by accidents.
\end{remark}

The following theorem states that it suffices to consider risk functionals on the  marginals only.

\begin{theorem}
Let $\Risk_\pi$ be a risk measure on the product space $\mathcal X\times\mathcal X$ and $X\colon\mathcal X\to\mathbb R$ a random variable. Then there is a marginal support set $\mathcal A_\prob$ such that
\[	
	\Risk_\prob(X)
	=	
	\sup\left\{ \evm{\prob}{Z_1\cdot X}\colon Z_1\in\mathcal{A}_\prob \right\}.
\]
The marginal support set is
\[
	\mathcal{A_\prob} =\left\{ \evm{\pi}{Z\mid \mathtt p_1 = \cdot}\colon Z \in \mathcal{A}_\pi\right\},
\]
where we view $\evm{\pi}{Z\mid \mathtt p_1 = \cdot}$ as a real-valued random variable on $\mathcal X = \mathrm{range}(\mathtt p_1)$.
The assertion holds analogously for $\Risk_\probalt$ with $\mathcal{A_\probalt} =\left\{ \evm{\pi}{Z\mid \mathtt p_2 = \cdot}\colon Z \in \mathcal{A}_\pi\right\}$.
\end{theorem}
\begin{proof}
Let $Z \in \mathcal A_\pi$.
By the Doob--Dynkin lemma there exists a measurable function $\psi\colon \mathcal X \to \mathbb R$ such that $ \evm{\pi}{Z\mid \mathtt p_1} = \psi \circ \mathtt p_1$ almost surely.
Thus, we set $Z_1(x) := \evm{\pi}{Z\mid \mathtt p_1 = x} = \psi(x)$ which is a real-valued random variable on $\mathcal X$ whereas $\evm{\pi}{Z\mid \mathtt p_1}$ is a real-valued random variable on $\mathcal X\times \mathcal X$.
By $\prob = \pi \circ \mathtt p_1^{-1}$ and the very definition of conditional expectation we have
\begin{align*}
	\evm{\prob}{Z_1}
	&
	=
	\evm{(\mathtt p_1)_*\pi}{\psi}
	= 
	\evm{\pi}{\psi(\mathtt p_1)}
	= 
	\evm{\pi}{\evm{\pi}{Z\mid \mathtt p_1}}
	=
	\evm{\pi}{Z}
	=
	1.
\end{align*}
Moreover, $Z_1 =  \evm{\pi}{Z\mid \mathtt p_1 = \cdot}\geq0$ almost surely, since $Z\geq0$ almost surely. 
Furthermore, as $X\circ \mathtt p_1$ is measurable with respect to $\sigma(\mathtt p_1)$, we obtain
\begin{align*}
	\evm{\prob}{X\cdot Z_1}
	& = \evm{(\mathtt p_1)_*\pi}{X\cdot \psi}
	= \evm{\pi}{(X\circ \mathtt p_1) \cdot \evm{\pi}{Z\mid \mathtt p_1}}
	= \evm{\pi}{\evm{\pi}{(X\circ \mathtt p_1) \cdot Z\mid \mathtt p_1}}\\
	& = \evm{\pi}{(X\circ \mathtt p_1) \cdot Z}
\end{align*}
Set $\mathcal{A_\prob} =\left\{ \evm{\pi}{Z\mid \mathtt p_1 = \cdot}: Z \in \mathcal{A}_\pi\right\}$ to obtain the assertion of the theorem.
\end{proof}

Exploiting the connection of $\Risk_\prob$ and $\Risk_\probalt$ via $\Risk_\pi$, we obtain the following sensitivity result.

\begin{theorem}[Sensitivity with respect to the probability measure] \label{thm:12}
	Let $\prob$, $\probalt \in \mc P_p(\mc X)$ and $\pi \in \Pi(\prob, \probalt)$.
	Further, assume that the mapping $X\colon \mc X \to \mathbb R$ is H\"older-continuous with exponent $\beta\in(0,1]$, i.e., $|X(x)-X(y)|\le C\cdot d_{\mc X}(x,y)^{\beta}$.
	Then
	\begin{equation}\label{eq:16}
		\left|\Risk_\prob(X) - \Risk_\probalt (X)\right|
		\le 
		C \ \sup_{Z\in\mathcal{A}_\pi} \evm{\pi}{Z^{p_\beta}}^{1/{p_\beta}}\  \evm{\pi}{d_{\mc X}^{\beta\, p}}^{1/p},
	\end{equation}
	where $p_{\beta}:=\frac p{p-\beta}$. 
	In particular, if $(\mc X,d_{\mc X})$ is complete and separable, choosing the optimal coupling $\pi^* \in \Pi(\prob, \probalt)$ for $d_p(\prob, \probalt)$, yields 
	\begin{equation}\label{eq:17}
		\left|\Risk_\prob(X)-\Risk_\probalt(X)\right|
		\le
		C \cdot  \sup_{Z \in {\mathcal A}_{\pi^*}} 
		\evm{\pi^*}{Z^{p_\beta}}^{1/{p_\beta}}\cdot d_p(\prob, \probalt)^{\beta}.
	\end{equation}
\end{theorem}
\begin{proof}
Let $Z\in\mathcal{A}_\pi$ be fixed. 
Note that
\begin{align*}
	\ev{Z\cdot X\circ \mathtt p_1}-\ev{Z\cdot X\circ \mathtt p_2} 
	&=
	\iint_{\mathcal X\times\mathcal X} Z(x,y)\cdot X(x)-Z(x,y)\cdot X(y)\,\pi(\mathrm dx,\mathrm dy)\\
	&=
	\iint_{\mathcal X\times\mathcal X} Z(x,y)\cdot\big(X(x)-X(y)\big)\,\pi(\mathrm dx,\mathrm dy)\\
	&\le 
	C\iint_{\mathcal X\times\mathcal X} d_{\mc X}(x,y)^{\beta}\cdot Z(x,y)\,\pi(\mathrm dx,\mathrm dy),
\end{align*}
as $Z\ge0$ $\pi$\nobreakdash-a.s\@. 
Taking the supremum among all $Z\in\mathcal{A}_\pi$ on the right hand side gives
\[
	\ev{Z\cdot X\circ \mathtt p_1}-\ev{Z\cdot X\circ \mathtt p_2} 
	\le 
	C \cdot \Risk_\pi\big(d_{\mc X}^{\beta}\big).
\]
Employing the definition~\eqref{eq:RP} and the optimal $Z$ there it follows that
\[
	\Risk_\prob(X)-\ev{Z\cdot X\circ \mathtt p_2}
	\le 
	C \cdot \Risk_\pi\!\big(d_{\mc X}^{\beta}\big),
\]
and now with~\eqref{eq:RQ} further that
\begin{equation} \label{eq:No}
	\Risk_\prob(X) - \Risk_\probalt(X)
	\le 
	C \cdot \Risk_\pi\!\big(d_{\mc X}^{\beta}\big).
\end{equation}
Furthermore, note that $\nicefrac p{\beta}$ and $p_{\beta}$ are H\"older conjugate exponents, $\frac1{p_{\beta}}+\frac1{\nicefrac p{\beta}}=1$.
With H\"older's inequality we thus obtain
\[
	\Risk_\pi\big(d^{\beta}\big)
	=
	\sup_{Z\in\mathcal{A}_\pi}\evm{\pi}{Z\cdot d_{\mc X}^{\beta}}
	\le
	\sup_{Z\in\mathcal{A}_\pi} 
	\evm{\pi}{Z^{p_\beta}}^{1/{p_\beta}} \evm{\pi}{d_{\mc X}^{\beta p}}^{1/p}.
\]
Interchanging the roles of $\prob$ and $\probalt$ yields the absolute value in~\eqref{eq:No} and therefore~\eqref{eq:16}.
The statement~\eqref{eq:17} then follows from~\eqref{eq:16} by noting that
\[
		\evm{\pi^*}{d_{\mc X}^{\beta p}}^{1/p}
		=
		d_{\beta p}(\prob,\probalt)^{\beta}
		\leq 
		d_{p}(\prob,\probalt)^{\beta}.
\]
\end{proof}
Theorem~\ref{thm:12} involves the supremum $\sup_{Z\in\mathcal{A}_\pi} \left\Vert Z\right\Vert _{\pi, p_{\beta}}$. This quantity is indeed finite for many important risk measures and in what follows we give explicit expressions for this bound:
\begin{itemize}
	\item 
	The optimal random variable $Z^*$ for the average value-at-risk satisfies $\pi\left(Z^*=\frac1{1-\alpha}\right) = 1-\alpha$ and $\pi(Z=0)=\alpha$, so that $\evm{\pi}{Z^{q}}^{1/{q}} \leq \left(\frac1{1-\alpha}\right)^{1-\frac1q} \leq \frac1{1-\alpha}$ for any $q\ge 1$. 
	\item 
	For the spectral risk functional it follows from~\eqref{eq:Spectral} that it is enough to require $\sigma\in L^{p_\beta}([0,1])$, as 
	$\evm{\pi}{Z^{q}}^{1/{q}} = \|\sigma\|_{L^q([0,1])}$.
	\item 
	Bounds for the entropic risk functional are elaborated in 
	Ahmadi-Javid and Pichler	
	\cite{AhmadiPichler} so that $\evm{\pi}{Z^{q}}^{1/{q}} \le \max\left(1,\ \frac{q-1}{\log\frac1{1-\alpha}}\right)$ in this case.
\end{itemize}

We see that for these risk measures the supremum $\sup_{Z\in\mathcal{A}_\pi} \left\Vert Z\right\Vert _{\pi, p_{\beta}}$ can be bounded independently of the coupling $\pi$.
We can then insert the corresponding bound into~\eqref{eq:16} and take the infimum over any coupling $\pi \in \Pi(\prob,\probalt)$ on the right-hand side.
Thus, for these risk measures and H\"older-continuous $X\colon \mathcal X \to \mathbb R$ with H\"older-exponent $\beta\in(0,1]$ there exists a constant $C = C(p, \beta, \Risk)<\infty$ such that
\begin{equation}\label{eq:18}
	\left| \Risk_\prob(X) -\Risk_\probalt(X) \right|
	\le
	C\, d_p(\prob, \probalt)^{\beta}
\end{equation}
for the case of a general (e.g., nonseparable or incomplete) metric space $(\mc X,d_{\mc X})$. 

\begin{remark} \label{rem:LocalLipQoI}
	We note that Theorem~\ref{thm:12} can be extended to only locally H\"older $X\colon \mc X\to\mathbb{R}$, i.e., $X$ for which there exists a nondecreasing $C_X\colon [0,\infty) \to [0,\infty)$ such that $|X(x) - X(y)| \leq C_X(r) \, d_{\mc X}(x,y)^\beta$ for all $x,y$ belonging to a ball of radius $r>0$ around a reference center $x_0 \in \mc X$.
	If $C_X$ then satisfies an integrability condition w.r.t.~$\prob$ and $\probalt$ similar to~\eqref{eq:Lip_cond_meas}, we can modify Theorem~\ref{thm:12} accordingly, following the main ideas in the proof of Theorem~\ref{theo:LocLip}.
\end{remark}

\subsection{Application to Random PDEs} 
We can now combine Theorem~\ref{thm:12} with the results from Section~\ref{sec:UQ_sens} on the Wasserstein sensitivity of the solution of random PDEs~\eqref{eq:randomPDE} to obtain sensitivity results for risk functionals of Lipschitz continuous quantities of interest applied to the random solution $u$ of~\eqref{eq:randomPDE}.
In particular, by the discussion following Theorem~\ref{thm:12} we obtain the next result.

\begin{corollary}\label{thm:UQ}
	Let 
	$\phi\colon \mc U\to\mathbb{R}$ have Lipschitz constant $C_\phi$ (cf.~\eqref{eq:Phi}) and let $\Risk$ denote a spectral risk functional or the entropic value-at-risk.  
	Then the following hold:
	\begin{enumerate}
		\item
		If $\prob,\probalt \in \mathcal P_p(\mathcal X)$, $p\geq 1$, and $\map\colon \mathcal X \to \mathcal U$ satisfies Assumption~\ref{assum:Lip}, then
		\begin{equation*}
			\left|\Risk_\prob(\phi \circ \mathcal S)
				- \Risk_\probalt(\phi \circ \mathcal S)\right|
			\le 
			C_{\Risk}\, C_\phi\, C_{\mc S} \, d_p(\prob,\probalt)^\beta
		\end{equation*}
		with a constant $C_{\Risk} < +\infty$ depending only on $p$, $\beta$, and $\Risk$.
		\item
		If $\map\colon \mcX \to \mc U$ satisfies Assumption~\ref{assum:LocLip} and $\prob, \probalt \in \mc P_{2p}(\mc X)$, $p\geq 1$, satisfy~\eqref{eq:Lip_cond_meas}, then
		\begin{equation*}
			\left|\Risk_\prob(\phi \circ \mathcal S)
				- \Risk_\probalt(\phi \circ \mathcal S)\right|
			\le 
			2C_{\Risk}\, C_\phi\, C^{\nicefrac 1{(2p)}} \cdot d_{2p}(\prob,\probalt)^\beta
		\end{equation*}
		with a constant $C_{\Risk} < +\infty$ depending only on $p$, $\beta$, and $\Risk$
		and $C$ as in~\eqref{eq:Lip_cond_meas}.
	\end{enumerate}
\end{corollary}

Obviously, Corollary~\ref{thm:UQ} can be extended to general risk functionals given that $(\mc X, d_{\mc X})$ is a complete and separable metric space and by replacing $C_{\Risk}$ with $\sup_{Z\in\mathcal{A}_{\pi^*}} \left\Vert Z\right\Vert _{\pi^*, p_{\beta}}$ as occuring in~\eqref{eq:17}.
Besides that one can also weaken the assumptions on $\phi$ to being only H\"older continuous or even locally H\"older continuous (following Remark~\ref{rem:LocalLipQoI}).
However, then also the Wasserstein distances $d_p(\prob,\probalt)^\beta$ and $d_p(\prob,\probalt)^\beta$ on the right-hand side have to be adapted accordingly.

\paragraph{The elliptic problem~\eqref{eq:BVP}}
We apply now Corollary~\ref{thm:UQ} to the solution operator of the elliptic problem~\eqref{eq:BVP} 
and combine this with the results in Section~\ref{sec:StrawberryFields}.
To this end, let $\log a$ be as in Proposition~\ref{theo:Trunc_Banach} with almost surely bounded $\xi_k$ and let
\[
	f = f_0 + \sum_{k=1}^\infty \widetilde \sigma_{k}\,\widetilde{\xi}_{k}\,f_{k}
\]
with uncorrelated mean-zero $\widetilde \xi_k$, $\var[\widetilde\xi_k]=1$, which are stochastically independent of the $\xi_k$.
Then, for any Lipschitz continuous $\phi\colon H_0^1(D)\to \bbR$ with Lipschitz constant $C_\phi$ and $\Risk$ a spectral risk functional or the entropic value-at-risk, we have for the distributions $\prob$, $\prob_K \in \mc P(L_+^\infty(D) \times L^2(D))$ of $(a, f)$ and $(a_K, f_K)$, respectively, that
\begin{align}
	\left|\Risk_{\prob}(\phi\circ\mathcal{S}) - \Risk_{\prob_K}(\phi\circ\mathcal{S})\right|
	\leq 
	2C_{\Risk}\, C_\phi\, C^{\nicefrac 1{(2p)}}
	\sqrt{
	\left(\sum_{k>K} |\sigma_k| \|a_k\|_{L^\infty(D)} \right)^2 
	+ 
	\sum_{k>K} \widetilde \sigma_{k}^2 \|f_{k}\|^2_{L^2(D)} 
	}
\end{align}
with $C_{\Risk}$, $C < +\infty$.
This truncation result extends to unbounded random coefficients $\xi_k$, i.e., $\xi_k \sim \mc N(0,1)$, provided estimates as referred to in Remark~\ref{rem:Charrier} are available.

\section{Summary} \label{sec:Summary}

We have derived stability results for uncertainty propagation for a stationary diffusion problem with random coefficient and source functions against perturbation of their probability distribution.
This result generalizes to the pushforward of probability measures under locally Hölder mappings and can therefore be applied to any differential equation with such a solution operator.
Quantities of interest deriving from the solution of such a random PDE inherit its randomness, and we have employed risk functionals to quantify the impact of their random behavior.
In order to capture deviations in (input) probability measures we used the Wasserstein distance which has a natural relation to Lipschitz continuous forward maps in view of the Kantorovich--Rubinstein theorem.
However, the solution operator of the problem considered is merely locally, not globally Lipschitz.
Employing new results on locally Hölder forward maps we were able to bound the deviations in the random solution and in risk functionals of derived quantities of interest. 
We applied our analysis to the common case of lognormal diffusion with a Mat\'ern covariance kernel exploiting some classical boundedness results on Gaussian processes as well as to the usual approximation of input random fields by truncated series expansions.

With these basic stability results for uncertainty propagation, one can apply the presented analysis to other PDE models with locally Hölder solution operators using the general results of Section~\ref{sec:Wasserstein_Sens}.
Moreover, for practical purposes it is helpful to have estimates on the Wasserstein distance of the input distributions, e.g., to be able to bound the distance of Gaussian random fields with different Mat\'ern covariance kernels by the difference in the Mat\'ern parameters.
In this way, the effects of the statistical estimation of these parameters on the outcome of uncertainty propagation for lognormal diffusion could be evaluated using our results.

\appendix

\section{Proof of Theorem~\ref{theo:lognormal}} \label{app:GRF}
In what follows we provide several lemmas which we combine to prove Theorem~\ref{theo:lognormal}, see the very end of this appendix. 
However, the argumentation and the statements of the lemmas are deliberately presented for arbitrary continuous Gaussian random fields, and we focus on Mat\'ern covariance functions only in the last part.
In particular, we consider pathwise continuous Gaussian random fields $g\colon D \times \Omega$ with continuous mean function $m(\cdot) := \ev{g(\cdot)} \in C(D)$ and continuous covariance function $c\in C(D\times D)$, $c(x,y):= \cov\big(g(x),g(y)\big)$. 
We denote the resulting Gaussian distribution on $C(D)$ of such a Gaussian random field by $\mathcal N(m,c) \in \mathcal P(C(D))$.

We study now the following question: for which sets $\mathcal M \subseteq C(D)$ and $\mathcal C \subseteq C(D\times D)$ of continuous mean and covariance functions can we ensure that for a given $\beta>0$ we have
\[
	\sup_{\prob \in \mc G(\mc M, \mc C)} \int \exp(\beta\,\|g\|_{C(D)}) \ \mathrm dP < \infty,
	\qquad
	\mathcal G(\mathcal M, \mathcal C)
	\coloneqq \left\{\mathcal N(m,c)\colon m \in \mathcal M, c \in \mathcal C \right\}\ ?
\]
By \emph{Fernique's theorem} we know that each single exponential moment above exists for arbitrary $\beta >0$.
However, the uniform boundedness over a given set $\mc G(\mc M, \mc C)$ is harder to ensure.
To this end, we apply the well-known \emph{Borell-TIS inequality} for Gaussian measures on Banach spaces \cite[Chapter~3]{LedouxTalagrand2002}: let~$g$ be a centered Gaussian process on a compact set $D \subseteq\mathbb R^d$ which is pathwise continous, then $\ev{\|g\|_{C(D)}}<\infty$ and with $\prob$ denoting its distribution on $C(D)$ we have
\begin{equation}\label{equ:Borell}
	\prob \left( \left| \|g\|_{C(D)} - \ev{\|g\|_{C(D)}} \right| \geq r\right) \leq 2 \exp(-r^2/2\sigma^2),
\end{equation}
where $\sigma^2 := \sup_{x\in D} \ev{g^2(x)} = \sup_{x\in D} \var(g(x))$.
This yields the following result.
\begin{lemma}\label{lem:ExpMomGauss}
Let $\mathcal G = \mathcal G(\mathcal M, \mathcal C)$, where $\mathcal M \subseteq C(D)$ and $\mathcal C \subseteq C(D\times D)$, satisfy the following conditions:
	\begin{enumerate}
		\item The sets $\mc M$ and $\mc C$ are bounded, i.e., there exist radii $r_m$, $r_c < \infty$ such that $\|m\|_{C(D)} \leq r_m$ for all $m \in \mathcal M$ and $\|c\|_{C(D\times D)} \leq r_c$ for all $c \in \mathcal C$.
		
		\item There exists a constant $s < \infty$ such that $\evm{\prob}{\|g\|_{C(D)}} \leq s$ for all $\prob \in \mathcal G(\{0\}, \mathcal C)$.
	\end{enumerate}
	Then, for any $0<\beta<\infty$, there exists a constant $K_\beta = K_\beta(r_m,r_c,s)<\infty$ such that
	\[
		\evm{\prob}{\exp(\beta \|g\|_{C(D)})}
		\leq
		K_\beta,
		\qquad
		\forall P \in \mathcal G(\mc M, \mc C).
	\]
\end{lemma}
\begin{proof}
Consider first an arbitrary $\prob = \mathcal N(m,c) \in \mathcal G(\mathcal M, \mathcal C)$.
Then we have
\[
	\evm{\prob}{\exp(\beta \|g\|_{C(D)})}
	\leq
	\exp(\beta r_m)\
	\evm{\prob}{\exp(\beta\, \|b-m\|_{C(D)})}.
\]
Thus, we may restrict ourselves in the following to an arbitrary centered $\prob = \mathcal N(0,c) \in \mathcal G(\{0\}, \mathcal C)$.
For convenience we set $M_\prob := \evm{\prob}{\|g\|_{C(D)}}$ and obtain by the second assumption and the Borell-TIS inequality
\[
	\prob\left(\{g\in C(D)\colon \left| \|g\|_{C(D)} - M_\prob \right| \geq r\}\right) \leq 2 \exp(-r^2/2r_c).
\]
This implies that
\[
	\prob\left(\{g\in C(D)\colon \|g\|_{C(D)} \geq M_\prob + r \}\right) \leq 2 \exp(-r^2/2r_c)
\]
which can now be used as follows:
\begin{align*}
	\evm{\prob}{\exp(\beta \|g\|_{C(D)})}
	& \leq
	\exp(\beta M_\prob) + \int_{g\colon \|g\|_{C(D)} \geq M_\prob} \exp(\beta \|g\|_{C(D)}) \ \mathrm d\prob\\
	& \leq
	\exp(\beta M_\prob) + \sum_{n=0}^\infty \e^{\beta (M_\prob+n+1)} \ \prob(\{g\in C(D)\colon n \leq \|g\|_{C(D)} - M_\prob < n+1\})\\
	& \leq
	\exp(\beta M_\prob) \left( 1 + \sum_{n=0}^\infty \e^{\beta (n+1)} \prob(\{g\in C(D)\colon M_\prob + n \leq \|g\|_{C(D)}\}) \right)\\
	& \leq
	\e^{\beta s} \left( 1 + 2\sum_{n=0}^\infty \e^{\beta (n+1) - n^2/2r_c } \right) < \infty,
\end{align*}
which proves the assertion.
\end{proof}

Actually, Fernique's theorem can be proven similarly to Lemma~\ref{lem:ExpMomGauss}. 
It remains to ensure a uniform bound of $\evm{\prob}{\|g(x)\|_{C(D)}}$, $\prob \in \mc G(\{0\}, \mc C)$ by imposing suitable conditions on $\mc C$.
Studying $\ev{\sup_{x\in D} g(x)}$ has a long history in probability theory. 
We refer to Talagrand \cite{Talagrand2006} for a comprehensive discussion and exploit \emph{Dudley's entropy bound}, a classical result.

Consider a centered Gaussian process $g$ with distribution $\prob = \mathcal N(0,c)$ on $C(D)$ and assume a metric $d\colon D\times D\to[0,\infty)$ satisfying
\begin{equation}\label{equ:Dudley_cond}
	\forall x,y \in D \ \forall r > 0\colon \prob(|g(x)-g(y)| \geq r) \leq 2 \exp\left(- \frac{r^2}{2d^2(x,y)}\right).
\end{equation}
For Gaussian processes such a metric is, for instance, given by
\[
	d_c(x,y) := \left(\ev{|g(x)-g(y)|^2} \right)^{1/2} = \sqrt{\var(g(x)-g(y) )} = \sqrt{c(x,x) + c(y,y) - 2c(x,y)}.
\]
Now, \emph{Dudley's entropy bound} \cite{Talagrand2006} for Gaussian processes $g$ with distribution $\prob = \mathcal N(0,c) \in \mc P(C(D))$ states that if~\eqref{equ:Dudley_cond} holds for a metric $d$, then we have
\begin{equation}\label{equ:Dudley_bound}
	\evm{\prob}{\|g\|_{C(D)}} \leq K \int_0^\infty \sqrt{\log \mc N(D,d_c,r)}\,\mathrm dr,
\end{equation}
where $K<\infty$ denotes a (universal) constant and $\mc N(D,d,r)$ are the \emph{covering numbers} of $D$ with respect to the metric $d$, i.e.,
\[
	\mc N(D,d,r)
	:=
	\inf \left\{n \in \mathbb N\colon \exists x_1, \ldots x_n \in D \text{ such that } D \subseteq \bigcup_{j=1}^n B_r^d(x_j)\right\},
\]
where $B_r^d(x) := \{y \in D\colon d(x,y)\leq r\}$ are balls in $D$ of radius $r$ with respect to $d$.
This gives rise to the following lemma.

\begin{lemma}\label{lem:Dudley}
Consider $\mathcal G(\{0\}, \mathcal C)$, $\mc C \subseteq C(D\times D)$.
If there exists a metric $d \colon D\times D \to [0,\infty)$ such that
\[
	c(x,x) + c(y,y) - 2c(x,y) \leq d^2(x,y) \qquad \forall x,y \in D\ \forall c \in \mathcal C 
\]
and
\[
	\int_0^\infty \sqrt{\log \mc N(D,d,r)}\,\mathrm dr < \infty,
\]
then
\[
	\sup_{\prob \in \mathcal G(\{0\}, \mathcal C)} \evm{\prob}{\|g\|_{C(D)}} < \infty.
\]
\end{lemma}

We now state a result for subsets of Mat\'ern covariance functions, as introduced in Subsection~\ref{sec:Wasserstein_Sens}, which in combination with the previous two lemmas proves Theorem~\ref{theo:lognormal} in Subsection~\ref{sec:Wasserstein_Sens}.

\begin{lemma} \label{lem:appMatern}
Consider the following class of Mat\'ern covariance functions:
\[
	\mathcal C 
	=
	\mathcal C(\sigma_{\max}, \rho_{\min},k_{\max})
	:= \left\{ c_{\sigma^2,\rho,k+\frac 12}\colon \sigma \leq \sigma_{\max}, \rho \geq \rho_{\min}, k \in \{0,\ldots,k_{\max}\} \right\},
\]
where $\sigma_{\max}<\infty, k_{\max} \in \mathbb N_0$ and $\rho_{\min} > 0$.
Then, we have
\[
	c(x,x) + c(y,y) - 2c(x,y) \leq d^2(x,y) \qquad \forall c \in \mathcal C, 
\]
where
\[
	d^2(x,y)\coloneqq 2\sigma^2_{\max} \left( 1 - \exp\left(- \frac{\sqrt{2k_{\max}+1}}{\rho_{\min}} |x-y|\right) \right),
\]
which is the associated metric $d_{c_\star}$ for $c_\star := c_{\sigma^2_{\max}, \widehat \rho, \frac 12}$ with $\widehat \rho := \frac{\rho_{\min}}{\sqrt{2k_{\max}+1}}$.
Moreover, for any bounded domain $D\subseteq \mathbb R^n$ we have for this metric that
\[
	\int_0^\infty \sqrt{\log \mc N(D,d,r)} \mathrm dr < \infty.
\]
\end{lemma}
\begin{proof}
Since $c_{\sigma^2,\rho, k+\frac 12}(x,x) = \sigma^2$, we have
\[
	c_{\sigma^2, \rho, k+\frac12}(x,x) + c_{\sigma^2, \rho, k+\frac12}(y,y) - 2c_{\sigma^2, \rho, k+\frac12}(x,y) = 2\sigma^2 (1 - c_{1, \rho, k+\frac12}(x,y)).
\]
Thus, the first assertion follows by
\[
	c_{1,\rho, k+\frac 12}(x,y)
	\geq 
	\exp\left(- \frac{\sqrt{2k+1}}{\rho} |x-y| \right)
	\geq
	\exp\left(- \frac{\sqrt{2k_{\max}+1}}{\rho_{\min}} |x-y|\right),
\]
where obtain the first inequality by using the explicit expression~\eqref{equ:Matern} and bounding
\begin{align*}
	\frac{k!}{(2k)!} 
	\sum_{i=0}^k \frac{(k+i)!}{i!(k-i)!} \left(2\frac{\sqrt{2k+1}}{\rho} |x-y| \right)^{k-i}
	\geq
	\frac{k!}{(2k)!}  \cdot
	\frac{(k+k)!}{k!(k-k)!} \left(2\frac{\sqrt{2k+1}}{\rho} |x-y| \right)^{k-k}
	= 1.	
\end{align*}
For the second assertion we 
introduce the auxilliary metric
\[
	\widehat d(x,y) :=  \sqrt{1 - \exp\left(- \frac{|x-y|}{\widehat \rho}\right)}
\]
with $\widehat \rho := \frac{\rho_{\min}}{\sqrt{2k_{\max}+1}}$ and notice that $d(x,y) = \sqrt 2 \sigma_{\max} \widehat d(x,y)$ for $x,y\in D$.
Hence, concerning balls we have for any radius $r>0$ that $B^d_r(x) = B^{\widehat d}_{r/\sqrt 2 \sigma_{\max}}(x)$ and, therefore, 
$\mc N(D,d,r) = \mc N(D, \widehat d, r/\sqrt{2} \sigma_{\max})$.
Thus, by a change of variables argument, we get 
\[
	\int_0^\infty \sqrt{\log \mc N(D,d,r)} \mathrm dr
	=
	\sqrt{2} \sigma_{\max} \int_0^\infty \sqrt{\log \mc N(D,\widehat d,r)} \mathrm dr.
\]
By the definition of $\widehat d$ we have that $\widehat d$ is bounded by $1$. 
Since for all $r\geq \diam_{\widehat d}(D)$ with $\diam_{\widehat d}(D) := \sup_{x,y\in D}\widehat d(x,y)$ denoting the diameter of $D$ w.r.t.~$\widehat d$ we have $\mc N(D,\widehat d, r) = 1$ we get
\[
	\int_0^\infty \sqrt{\log \mc N(D,\widehat d,r)}\ \mathrm dr
	=
	\int_0^1 \sqrt{\log \mc N(D,\widehat d,r)}\ \mathrm dr.
\]
Moreover, we obtain for the Euclidean norm $|\cdot|$ (or rather the corresponding distance) in $\mathbb R^n$ that 
\[
	B^{\widehat d}_{r}(x)
	=
	B^{|\cdot|}_{\widehat \rho \log(1/(1 - r^2))}(x)
	\qquad
	\forall r\in[0,1)\
	\forall x\in D,
\]
which yields $\mc N(D, \widehat d, r) = \mc N(D, |\cdot|, \widehat \rho \log((1-r^2)^{-1}))$.
Therefore, we next estimate $\mc N(D, |\cdot|, r)$ for $r\in[0,\infty)$.
We do so by embedding $D$ in a cube $\widehat D$ of edge length $\diam(D) := \sup_{x,y\in D}|x-y|$ and simply use $\mc N(D, |\cdot|, r) \leq \mc N(\widehat D, |\cdot|, r)$.
Now, since in every (Euclidean) ball of radius $r$ in $\mathbb R^n$ we can insert a smaller cube of edge length $a = 2r/\sqrt n$, we can bound the covering number $\mc N(\widehat D, |\cdot|, r)$ by the number of cubes of edge length $a$ covering $\widehat D$, i.e.,
\[
	\mc N(D, |\cdot|, r) 
	\leq 
	\mc N(\widehat D, |\cdot|, r)
	\leq
	\left(\frac{\diam(D)}a \right)^n
	=
	\left(\frac{\sqrt n \diam(D)}{2r} \right)^n.
\]
Thus, we obtain for any $r\in[0,1)$
\begin{align*}
	\mc N(D, \widehat d, r)
	& 
	= \mc N(D, |\cdot|, \widehat \rho \log((1-r^2)^{-1}))
	\leq
	\left\lceil \frac{\sqrt n \diam(D)}{2\widehat \rho \log((1-r^2)^{-1})} \right\rceil^n
	=
	\kappa \left\lceil \frac1{\log((1-r^2)^{-1})}\right\rceil^n
\end{align*}
where we defined $\kappa := \left\lceil \sqrt n \diam(D)/(2\widehat \rho)\right\rceil^n$.
Hence,
\begin{align*}
	\int_0^\infty \sqrt{\log \mc N(D,d,r)} \mathrm dr
	& =
	\sqrt{2} \sigma_{\max}  \int_0^1 \sqrt{\log \mc N(D,\widehat d,r)}\ \mathrm dr\\
	& 
	\leq
	\sqrt{2} \sigma_{\max} \left(\sqrt{\log\kappa} + \sqrt{n} \int_0^1 \sqrt{\log\left( \left\lceil 1/\log( (1-r^2)^{-1}) \right\rceil\right) }\ \mathrm dr \right).
\end{align*}
Thus, it remains to show that $\int_0^1 \sqrt{\log\left( \left\lceil 1/\log( (1-r^2)^{-1}) \right\rceil\right) }\ \mathrm dr < \infty$.
To this end, we use
\[
	\frac 1{\log((1-r^2)^{-1})} \leq \frac 1{r^2}, \qquad r\in (0,1),
\]
which can be verified by a simple calculation. Hence,
\[
	\int_0^1 \sqrt{\log\left( \left\lceil 1/\log( (1-r^2)^{-1}) \right\rceil\right) }\ \mathrm dr
	\leq
	\int_0^1 \sqrt{\log\left( \lceil r^{-2} \rceil \right) }\ \mathrm dr.
\]
Furthermore, we have $\lceil r^{-2} \rceil = 1$ for $r \geq 1$ and 
\[
	\left\lceil r^{-2} \right\rceil
	=
	i
	\quad
	\text{iff}
	\quad
	\frac 1{\sqrt{i}} < r \leq \frac 1{\sqrt{i-1}}
\]
for any $2\leq i \in \mathbb N$.
Therefore,
\begin{align*}
	\int_0^1 \sqrt{\log\left( \lceil r^{-2} \rceil \right) }\ \mathrm dr
	& 
	=
	\sum_{i=2}^\infty
	\sqrt{\log(i)} \left|\frac 1{\sqrt{i-1}} - \frac 1{\sqrt{i}}\right|
	\leq
	\sum_{i=2}^\infty
	\sqrt{\log(i)} \ \frac 1{\sqrt{i^3}}	
	< \infty,
\end{align*}
where the first inequality can be verified again by a straightforward calculation and the convergence of the series follows by a simple comparison arguement.
This concludes the proof. 
\end{proof}

\begin{proof}[Proof of Theorem~\ref{theo:lognormal}]
In order to prove Theorem~\ref{theo:lognormal}, it suffices to show that the class of Gaussian measures specified in Theorem~\ref{theo:lognormal} satisfies the assumption of Lemma~\ref{lem:ExpMomGauss}.
Now, condition $(i)$ of Lemma~\ref{lem:ExpMomGauss} is ensured by the assumptions of Theorem~\ref{theo:lognormal}, since $\|c_{\sigma^2,\rho,k+\frac 12}\|_{C(D\times D)}=\sigma^2$.
Condition $(ii)$ of Lemma~\ref{lem:ExpMomGauss} follows by combining Lemma~\ref{lem:Dudley} and Lemma~\ref{lem:appMatern}.
This concludes the proof.
\end{proof}

\section{Explicit Contruction of (Coupled) Risk Measures} \label{app:Risk}

Let $\pi$ be a coupling with marginals $\prob$ and $\probalt$. 
In what follows we explicitly construct a risk functional $\Risk_\pi$ so that 
\[
   \Risk_\pi(X\circ \mathtt p_1) = \mathcal P_\prob(X)
   \quad \text{ and }\quad 
   \Risk_\pi(X\circ \mathtt p_2) = \mathcal P_\probalt(X).
\]
For ease of exposition we define the measure $Z_1\prob$ ($Z_2\probalt$, resp.\@) with density $Z_1$ ($Z_2$, resp.\@) by 
\[
	Z_1\prob(B):=\int_B Z\,\mathrm dP\qquad (Z_2\probalt(B)
	:=
	\int_B Z_2\,\mathrm d\probalt, \text{ resp.}).
\]

\begin{proposition}
Let the risk functionals $\Risk_\prob$ and $\Risk_\probalt$ have support sets $\mathcal A_\prob$ and $\mathcal A_\probalt$. 
Then the set
\[
	\mc A_\pi
	:=
	\left\{
	\vZ = \frac{\mathrm d \gamma}{\mathrm d \mathrm \pi}
	: 
	\gamma\in \Pi(Z_1\prob,\, Z_2\probalt),\ Z_1 \in \mc A_\prob,\ Z_2 \in \mc A_\probalt 
	\right\}
\]
is the support set of a risk functional $\Risk_\pi$ satisfying $\Risk_\pi(X\circ \mathtt p_1) = \Risk_\prob(X)$ and $\Risk_\pi(X\circ \mathtt p_2) = \Risk_\probalt(X)$.
\end{proposition}
\begin{proof}
Let $\vZ \in \mc A_\pi$ be arbitrary and denote by $Z_1 \in \mc A_\prob$ and $Z_2\in\mathcal A_\probalt$ the two random variables such that $\vZ\pi \in \prod(Z_1\prob, Z_2\probalt)$, where $\vZ\pi(A\times B):=\iint_{A\times B} \vZ(x_1,x_2)\,\pi(\mathrm dx_1,\,\mathrm dx_2)$ is defined in analogy to $Z_1\prob$ and $Z_2\probalt$.
We then have
\begin{align*}
	\evm{\pi}{\vZ (X\circ \mathtt p_1)}
	&=
	\int_{\mcX\times\mcX} (X\circ \mathtt p_1)(x_1,x_2)\,(\vZ\pi)(\mathrm dx_1 ,\,\mathrm d x_2)
	=
	\int_{\mcX} X(x_1)\, Z_1\prob(\mathrm d x_1) =\evm{\prob}{Z_1X}.
\end{align*}
	
Further, for any $Z_1 \in \mc A_\prob$ it holds that $\vZ(x_1,x_2):=Z_1(x)\in\mathcal A_\pi$, as $Z_2(x_2)=1\in\mathcal A_\probalt$. 
We thus obtain that
\[
	\Risk_\pi(X\circ \mathtt p_1)
	=
	\sup_{\vZ \in \mc A_\pi} \evm{\pi}{\vZ\cdot(X\circ \mathtt p_1)}
	=
	\sup_{{\substack{Z_1 \in \mc A_1}}} \evm{\prob}{Z_1X}
	=
	\Risk_\prob(X).
\]
The statement for $\Risk_\probalt$ is proven analogously.
\end{proof}

\nocite{RockafellarUryasev2013}
\bibliographystyle{siamplain}	
\bibliography{LiteraturAlois, literature}

\begin{thebibliography}{10}

\bibitem{AhmadiPichler}
{\sc A.~Ahmadi-Javid and A.~Pichler}, {\em An analytical study of norms and
  {B}anach spaces induced by the entropic value-at-risk}, Mathematics and
  Financial Economics, 11 (2017), pp.~527--550,
  \url{https://doi.org/10.1007/s11579-017-0197-9}.

\bibitem{Artzner1999}
{\sc P.~Artzner, F.~Delbaen, J.-M. Eber, and D.~Heath}, {\em Coherent
  {M}easures of {R}isk}, Mathematical Finance, 9 (1999), pp.~203--228,
  \url{https://doi.org/10.1111/1467-9965.00068}.

\bibitem{Artzner1997}
{\sc P.~Artzner, F.~Delbaen, and D.~Heath}, {\em Thinking coherently}, Risk, 10
  (1997), pp.~68--71.

\bibitem{BabuskaEtAl2004}
{\sc I.~Babu\v{s}ka, R.~Tempone, and G.~Zouraris}, {\em {G}alerkin finite
  element approximations of stochastic elliptical partial differential
  equations}, SIAM Journal on Numerical Analysis, 42 (2004), pp.~800--825,
  \url{https://doi.org/10.1137/S0036142902418680}.

\bibitem{BabuskaEtAl2005}
{\sc I.~Babu\v{s}ka, R.~Tempone, and G.~E. Zouraris}, {\em Solving elliptic
  boundary value problems with uncertain coefficients by the finite element
  method: The stochastic formulation}, Computer Methods in Applied Mechanics
  and Engineering, 194 (2005), pp.~1251--1294,
  \url{https://doi.org/10.1016/j.cma.2004.02.026}.

\bibitem{Bolley2008}
{\sc F.~Bolley}, {\em Separability and completeness for the {W}asserstein
  distance.}, in S\'eminaire de Probabilit\'es XLI, C.~Donati-Martin,
  M.~\'Emery, A.~Rouault, and C.~Stricker, eds., vol.~1934 of Lecture Notes in
  Mathematics, Springer, Berlin, Heidelberg, 2008, pp.~371--377,
  \url{https://doi.org/10.1007/978-3-540-77913-1}.

\bibitem{BonitoEtAl2017}
{\sc A.~Bonito, A.~Cohen, R.~Devore, G.~Petrova, and G.~Welper}, {\em Diffusion
  coefficients estimation for elliptic partial differential equations}, SIAM J.
  Math. Anal., 49 (2017), pp.~1570--1592,
  \url{https://doi.org/10.1137/16M1094476}.

\bibitem{BrennerScott2008}
{\sc S.~Brenner and R.~Scott}, {\em The Mathematical Theory of Finite Element
  Methods}, Springer-Verlag, New York, 2008,
  \url{https://doi.org/10.1007/978-0-387-75934-0}.

\bibitem{Charrier2012}
{\sc J.~Charrier}, {\em Strong and weak error estimates for elliptic partial
  differential equations with random coefficients}, SIAM Journal on Numerical
  Analysis, 50 (2012), pp.~216--246, \url{https://doi.org/10.1137/100800531}.

\bibitem{Dupuis2013}
{\sc K.~Chowdhary and P.~Dupuis}, {\em Distinguishing and integrating aleatoric
  and epistemic variation in uncertainty quantification}, {ESAIM}: Mathematical
  Modelling and Numerical Analysis, 47 (2013), pp.~635--662,
  \url{https://doi.org/10.1051/m2an/2012038}.

\bibitem{Schultz2011}
{\sc S.~Conti, H.~Held, M.~Pach, M.~Rumpf, and R.~Schultz}, {\em Risk averse
  shape optimization}, {SIAM} Journal on Control and Optimization, 49 (2011),
  pp.~927--947, \url{https://doi.org/10.1137/090754315}.

\bibitem{Denneberg1989}
{\sc D.~Denneberg}, {\em Distorted probabilities and insurance premiums},
  Methods of Operations Research, 63 (1990), pp.~21--42.

\bibitem{GerberDeprez}
{\sc O.~Deprez and H.~U. Gerber}, {\em On convex principles of premium
  calculation}, Insurance: Mathematics and Economics, 4 (1985), pp.~179--189,
  \url{https://doi.org/10.1016/0167-6687(85)90014-9}.

\bibitem{DowsonLandau1982}
{\sc D.~C. Dowson and B.~V. Landau}, {\em The {F}r\'echet distance between
  multivariate normal distributions}, Journal of Multivariate Analysis, 12
  (1982), pp.~450--455.

\bibitem{Dupuis2016}
{\sc P.~Dupuis, M.~A. Katsoulakis, Y.~Pantazis, and P.~Plech\'{a}\v{c}}, {\em
  Path-space information bounds for uncertainty quantification and sensitivity
  analysis of stochastic dynamics}, {SIAM}/{ASA} Journal on Uncertainty
  Quantification, 4 (2016), pp.~80--111,
  \url{https://doi.org/10.1137/15m1025645}.

\bibitem{Follmer2004}
{\sc H.~F\"{o}llmer and A.~Schied}, {\em Stochastic Finance: An Introduction in
  Discrete Time}, de Gruyter Studies in Mathematics 27, Berlin, Boston: De
  Gruyter, 2004, \url{https://doi.org/10.1515/9783110218053},
  \url{http://books.google.com/books?id=cL-bZSOrqWoC}.

\bibitem{Geiersbach2019}
{\sc C.~Geiersbach and W.~Wollner}, {\em A stochastic gradient method with mesh
  refinement for {PDE} constrained optimization under uncertainty}, 2019,
  \url{http://arXiv.org/abs/1905.08650v1},
  \url{https://arxiv.org/abs/1905.08650v1}.

\bibitem{Gelbrich1990}
{\sc M.~Gelbrich}, {\em On a formula for the {L}2 {W}asserstein metric between
  measures on {E}uclidean and {H}ilbert spaces}, Mathematische Nachrichten, 147
  (1990), pp.~185--203, \url{https://doi.org/10.1002/mana.19901470121}.

\bibitem{GibbsSu2002}
{\sc A.~L. Gibbs and F.~E. Su}, {\em On choosing and bounding probability
  metrics}, International Statistical Review, 70 (2001), pp.~419--435.

\bibitem{GoettlichKnapp2020}
{\sc S.~G{\"o}ttlich and S.~Knapp}, {\em Uncertainty quantification with risk
  measures in production planning}, J. Math. Industry, 10 (2020),
  \url{https://doi.org/10.1186/s13362-020-00074-4}.

\bibitem{Pflug2016}
{\sc P.~Gross and G.~C. Pflug}, {\em Behavioral pricing of energy swing options
  by stochastic bilevel optimization}, Energy Systems, 7 (2016), pp.~637--662,
  \url{https://doi.org/10.1007/s12667-016-0190-z}.

\bibitem{Halkos_2019}
{\sc G.~E. Halkos and A.~S. Tsirivis}, {\em Value-at-risk methodologies for
  effective energy portfolio risk management}, Economic Analysis and Policy, 62
  (2019), pp.~197--212, \url{https://doi.org/10.1016/j.eap.2019.03.002}.

\bibitem{Harbrecht2008}
{\sc H.~Harbrecht}, {\em Analytical and numerical methods in shape
  optimization}, Mathematical Methods in the Applied Sciences, 31 (2008),
  pp.~2095--2114, \url{https://doi.org/10.1002/mma.1008}.

\bibitem{HeinkenschlossEtAl2020}
{\sc M.~Heinkenschloss, B.~Kramer, and T.~Takhtaganov}, {\em Adaptive
  reduced-order model construction for conditional value-at-risk estimation},
  SIAM/ASA J. Uncertainty Quantification, 8 (2020), pp.~668--692,
  \url{https://doi.org/10.1137/19M1257433}.

\bibitem{HeinkenschlossEtAl2018}
{\sc M.~Heinkenschloss, B.~Kramer, T.~Takhtaganov, and K.~Wilcox}, {\em
  Conditional-value-at-risk estimation via reduced-order models}, SIAM/ASA J.
  Uncertainty Quantification, 6 (2018), pp.~1395--1423,
  \url{https://doi.org/10.1137/17M1160069}.

\bibitem{KalmesPichler}
{\sc T.~Kalmes and A.~Pichler}, {\em On {B}anach spaces of vector-valued random
  variables and their duals motivated by risk measures}, Banach Journal of
  Mathematical Analysis, 12 (2018), pp.~773--807,
  \url{https://doi.org/10.1215/17358787-2017-0026}.

\bibitem{Kouri2016}
{\sc D.~P. Kouri and T.~M. Surowiec}, {\em Risk-averse {PDE}-constrained
  optimization using the conditional value-at-risk}, {SIAM} Journal on
  Optimization, 26 (2016), pp.~365--396,
  \url{https://doi.org/10.1137/140954556}.

\bibitem{Kouri2018}
{\sc D.~P. Kouri and T.~M. Surowiec}, {\em Existence and optimality conditions
  for risk-averse {PDE}-constrained optimization}, {SIAM}/{ASA} Journal on
  Uncertainty Quantification, 6 (2018), pp.~787--815,
  \url{https://doi.org/10.1137/16m1086613}.

\bibitem{LedouxTalagrand2002}
{\sc M.~Ledoux and M.~Talagrand}, {\em Probability in Banach Spaces},
  Springer-Verlag, Berlin Heidelberg, 2002.

\bibitem{LordEtAl2014}
{\sc G.~J. Lord, C.~E. Powell, and T.~Shardlow}, {\em An Introduction to
  Computational Stochastic {PDE}s}, Cambridge University Press, New York, 2014.

\bibitem{LuschgyPages2009}
{\sc H.~Luschgy and G.~Pag\`es}, {\em Expansions for {G}aussian processes and
  {P}arseval frames}, Electronic Journal of Probability, 14 (2009),
  pp.~1198--1221.

\bibitem{OdenDemkowicz2018}
{\sc J.~T. Oden and L.~Demkowicz}, {\em Applied Functional Analysis}, CRC
  Press, Boca Raton, FL, 3rd~ed., 2018.

\bibitem{RuszOgryczak}
{\sc W.~Ogryczak and A.~Ruszczy\'{n}ski}, {\em Dual stochastic dominance and
  related mean-risk models}, SIAM Journal on Optimization, 13 (2002),
  pp.~60--78, \url{https://doi.org/10.1137/S1052623400375075}.

\bibitem{Pflug2000}
{\sc G.~{\relax Ch}. Pflug}, {\em Some remarks on the {V}alue-at-{R}isk and the
  {C}onditional {V}alue-at-{R}isk}, in Probabilistic Constrained Optimization,
  S.~Uryasev, ed., vol.~49, Springer US, 2000, ch.~15, pp.~272--281,
  \url{https://doi.org/10.1007/978-1-4757-3150-7}.

\bibitem{Pichler2013a}
{\sc A.~Pichler}, {\em The natural {B}anach space for version independent risk
  measures}, Insurance: Mathematics and Economics, 53 (2013), pp.~405--415,
  \url{https://doi.org/10.1016/j.insmatheco.2013.07.005}.

\bibitem{Pichler2013b}
{\sc A.~Pichler}, {\em Premiums and reserves, adjusted by distortions},
  Scandinavian Actuarial Journal, 2015 (2013), pp.~332--351,
  \url{https://doi.org/10.1080/03461238.2013.830228}.

\bibitem{ShapiroAlois}
{\sc A.~Pichler and A.~Shapiro}, {\em Minimal representation of insurance
  prices}, Insurance: Mathematics and Economics, 62 (2015), pp.~184--193,
  \url{https://doi.org/10.1016/j.insmatheco.2015.03.011}.

\bibitem{Rachev}
{\sc S.~T. Rachev}, {\em Probability Metrics and the Stability of Stochastic
  Models}, John Wiley and Sons, West Sussex, England, 1991,
  \url{http://books.google.com/books?id=5grvAAAAMAAJ}.

\bibitem{Rockafellar1974}
{\sc R.~T. Rockafellar}, {\em Conjugate Duality and Optimization}, vol.~16,
  CBMS-NSF Regional Conference Series in Applied Mathematics. 16. Philadelphia,
  Pa.: SIAM, Society for Industrial and Applied Mathematics. VI, 74 p., 1974,
  \url{https://doi.org/10.1137/1.9781611970524}.

\bibitem{RockafellarRoyset2015}
{\sc R.~T. Rockafellar and J.~O. Royset}, {\em Engineering decisions under risk
  averseness}, ASCE-ASMEJournal of Risk and Uncertainty in Engineering Systems,
  Part A: Civil Engineering, 1 (2015), p.~04015003,
  \url{https://doi.org/10.1061/AJRUA6.0000816}.

\bibitem{RockafellarUryasev2000}
{\sc R.~T. Rockafellar and S.~Uryasev}, {\em Optimization of {C}onditional
  {V}alue-at-{R}isk}, Journal of Risk, 2 (2000), pp.~21--41,
  \url{https://doi.org/10.21314/JOR.2000.038}.

\bibitem{RockafellarUryasev2013}
{\sc R.~T. Rockafellar and S.~Uryasev}, {\em The fundamental risk quadrangle in
  risk management, optimization and statistical estimation}, Surveys in
  Operations Research and Management Science, 18 (2013), pp.~33--53,
  \url{https://doi.org/10.1016/j.sorms.2013.03.001}.

\bibitem{Ruszczynski2006}
{\sc A.~Ruszczy\'{n}ski and A.~Shapiro}, {\em Optimization of convex risk
  functions}, Mathematics of Operations Research, 31 (2006), pp.~433--452,
  \url{https://doi.org/10.1287/moor.1050.0186}.

\bibitem{RuszczynskiShapiro2009}
{\sc A.~Shapiro, D.~Dentcheva, and A.~Ruszczy\'{n}ski}, {\em Lectures on
  {S}tochastic {P}rogramming}, MOS-SIAM Series on Optimization, SIAM,
  third~ed., 2014, \url{https://doi.org/10.1137/1.9781611976595}.

\bibitem{Sprungk2020}
{\sc B.~Sprungk}, {\em On the local {L}ipschitz stability of {B}ayesian inverse
  problems}, Inverse Problems, 36 (2020), p.~055015,
  \url{https://doi.org/10.1088/1361-6420/ab6f43}.

\bibitem{Stuart2010}
{\sc A.~M. Stuart}, {\em Inverse problems: a {B}ayesian perspective}, Acta
  Numerica, 19 (2010), pp.~451--559.

\bibitem{Talagrand2006}
{\sc M.~Talagrand}, {\em The Generic Chaining}, Springer, 2006,
  \url{https://doi.org/10.1007/3-540-27499-5}.

\bibitem{Villani2003}
{\sc C.~Villani}, {\em Topics in {O}ptimal {T}ransportation}, vol.~58 of
  Graduate Studies in Mathematics, American Mathematical Society, Providence,
  RI, 2003, \url{https://doi.org/10.1090/gsm/058},
  \url{http://books.google.com/books?id=GqRXYFxe0l0C}.

\bibitem{Villani2009}
{\sc C.~Villani}, {\em Optimal transport, old and new}, vol.~338 of Grundlehren
  der Mathematischen Wissenschaften, Springer, Berlin, 2009.

\end{thebibliography}

\end{document}